\documentclass[a4paper,12pt,reqno]{amsart}

\usepackage{color}
\usepackage{amsmath}
\usepackage{amssymb}
\usepackage{amsfonts}
\usepackage{graphicx}
\usepackage{mathtools}
\usepackage[colorlinks]{hyperref}
\renewcommand\eqref[1]{(\ref{#1})} 
\usepackage{booktabs}
\usepackage{lscape}
\usepackage{braket}
\usepackage{mathrsfs}
\usepackage{comment}
\usepackage{enumitem}
\usepackage{stackrel}
\usepackage{mathtools,bm}
\usepackage{tikz}
\usepackage{tikz-cd}
\usepackage{circuitikz}
\usepackage{graphicx}
\usepackage{dsfont}
\usepackage{float}
\usepackage[thinlines]{easytable}
\usepackage{array}
\usepackage{booktabs}

\graphicspath{ {images/} }
\newcommand{\bbC}{{\Bbb C}}

\newcommand{\bbN}{{\Bbb N}}
\newcommand{\bbR}{{\Bbb R}}

\newcommand{\bbZ}{{\Bbb Z}}

\newcommand{\bbS}{{\Bbb S}}

\newcommand{\bbT}{{\Bbb T}}

\newcommand{\skernel}{\mathscr{A}}

\allowdisplaybreaks
\title[$L^p$-bounds in Safarov pseudo-differential calculus]{$L^p$-bounds in Safarov pseudo-differential calculus on manifolds with bounded geometry}

\author[S. G\'omez Cobos]{Santiago G\'omez Cobos}
\address{
	Santiago G\'omez Cobos:
	\endgraf
	Department of Mathematics: Analysis, Logic and Discrete Mathematics
	\endgraf
	Ghent University, Krijgslaan 281, Building S8, B 9000 Ghent
	\endgraf
	Belgium
	\endgraf
	{\it E-mail address} {\rm davidsantiago.gomezcobos@ugent.be}}

\author[M. Ruzhansky]{Michael Ruzhansky}
\address{
	Michael Ruzhansky:
	\endgraf
	Department of Mathematics: Analysis, Logic and Discrete Mathematics
	\endgraf
	Ghent University, Krijgslaan 281, Building S8, B 9000 Ghent
	\endgraf
	Belgium
	\endgraf
	and
	\endgraf
    School of Mathematical Sciences
    \endgraf
    Queen Mary University of London
    \endgraf
    United Kingdom
    \endgraf
	{\it E-mail address} {\rm michael.ruzhansky@ugent.be}}


\newtheoremstyle{theorem}
{10pt}          
{10pt}  
{\sl}  
{\parindent}     
{\bf}  
{. }    
{ }    
{}     
\theoremstyle{theorem}

\numberwithin{equation}{section}
\theoremstyle{plain} 
\newtheorem{thm}{Theorem}[section]
\newtheorem{prop}[thm]{Proposition}
\newtheorem{cor}[thm]{Corollary}
\newtheorem{lem}[thm]{Lemma}

\theoremstyle{definition}
\newtheorem{defn}[thm]{Definition}
\newtheorem{rem}[thm]{Remark}
\newtheorem{ex}[thm]{Example}
\newtheorem{as}[thm]{Assumption}

\newtheoremstyle{defi}
{10pt}          
{10pt}  
{\rm}  
{\parindent}     
{\bf}  
{. }    
{ }    
{}     
\theoremstyle{defi}

\begin{document}
 	\begin{abstract}
Given a smooth complete Riemannian manifold with bounded geometry $(M,g)$ and a linear connection $\nabla$ on it (not necessarily a metric one), we prove the $L^p$-boundedness of operators belonging to the global pseudo-differential classes $\Psi_{\rho, \delta}^m\left(\Omega^\kappa, \nabla, \tau\right)$ constructed by Safarov. Our result recovers classical Fefferman's theorem, and extends it to the following two situations: $\rho>1/3$ and $\nabla$ symmetric; and $\nabla$ flat with any values of $\rho$ and $\delta$. Moreover, as a consequence of our main result, we obtain boundedness on Sobolev and Besov spaces and some $L^p-L^q$ boundedness. Different examples and applications are presented. 
\end{abstract}
\dedicatory{Dedicated to the memory of Yuri Safarov (1958–2015)}
	\maketitle
	\tableofcontents

\section{Introduction}
Studying continuity of linear operators defined on Banach spaces is an absolute classic problem in mathematics. Such abstract problems have been useful when applied in physics and engineering, especially due to their close relationship with the theory of PDE's. For instance, studying estimation norms of operators on $L^p$ spaces, Sobolev spaces, Besov spaces, etc., becomes extremely handy when it comes to obtaining existence, regularity and approximation of solutions to different types of linear and nonlinear equations (see e.g \cite{bao, hor, tao, Taylor2, Taylor3}). In particular,  mathematicians have been actively investigating the cases of Fourier and spectral multipliers, and more generally pseudo-differential operators, since those operators provide a fundamental, accessible and diverse family of partial differential operators. Let us be more precise on what type of operators we are discussing. Usually those operators are entirely defined by means of a function $m$, which could have one or two variables, this means that there is a way (quantization procedure) to produce operators $P:=\operatorname{Op}(m)$ and we are interested on establishing conditions in $m$ such that we can guarantee the boundedness of $P$ on some given Banach spaces $X$ and $Y$ ($X$ and $Y$ could be the same space), i.e. to prove inequalities like
\begin{equation}\label{Objec}
    \|\operatorname{Op}(m) f\|_X\leq C \|f\|_Y.
\end{equation}
One important feature of this type of problems is the nature of the underlying space where the operators are defined, usually each of them has their own difficulties and particular tools to overcome them; for instance the presence or absence of a Fourier transform. Some of the classical spaces of interest are: Euclidean space $\bbR^n$, Lie groups $G$, manifolds $M$, etc. In order to get a feeling on what type of results we are looking for, let us state one of the most classical results in this topic, the H\"ormander-Mikhlin Theorem \cite{hor1, mikh}:
\begin{thm}
Let $m : \bbR^n \to \bbC$ satisfy the following condition: 
\begin{equation}\label{HorM}
    \left|\frac{\partial^\alpha}{\partial\xi^\alpha}m(\xi)\right|\leq A_\alpha |\xi|^{-|\alpha|}
\end{equation}
for all nonzero $\xi\in\bbR^n$ and for all $0\leq |\alpha|\leq \lfloor\frac{n}{2}\rfloor+1$. Then the multiplier 
\[
m(D)f(x)= \frac{1}{(2\pi)^n}\int_{\bbR^n} e^{ix\cdot \xi}\, m(\xi) \widehat{f}(\xi)\, d\xi
\]
is bounded from $L^p(\bbR^n)$ to $L^p(\bbR^n)$ for $1 < p < \infty$.
\end{thm}
The main concern of this article is the case of pseudo-differential operators, which are operators defined by functions $\sigma(x,\xi)$ (later on we will specify where the variables belong to), so we move on the discussion towards that direction. The problem consists of identifying conditions like \eqref{HorM}, but for the two variable dependent functions $\sigma(x,\xi)$. This led mathematicians to construct different kinds of classes of functions in which it is possible to control the growth of their objects in a good way, making it possible to obtain inequalities of the type \eqref{Objec}. For the case of $\bbR^n$ everything started with the pioneering works of Calderón-Zygmund \cite{calzyg} and Kohn-Nirenberg \cite{Kohn}, but probably the H\"ormander classes $S_{\rho,\delta}^m$ \cite{Hormander} are the most studied ones since they are general enough to obtain useful results, even though there are much more general classes which have these ones as a particular case. For now, let us focus on the classes $S_{\rho,\delta}^m$, so we briefly recall its definition, for more details see e.g. \cite{hor, shu}. Let $m,\rho,\delta$ be real numbers such that $0\leq\delta\leq 1$, $0\leq\rho\leq 1$, and let $X$ be an open set of $\bbR^n$. The class $S_{\rho,\delta}^m(X,\bbR^n)$ consists of functions $\sigma\in C^\infty(X\times\bbR^n)$ such that for any multi-indices $\alpha,\beta$ and any compact set $K\subset X$ there exists a constant $C_{\alpha,\beta,K}$ such that 
\begin{equation}\label{defrn}
    |\partial_\xi^\alpha\partial_x^\beta \sigma(x,\xi)|\leq C_{\alpha,\beta,K} (1+|\xi|)^{m-\rho|\alpha|+\delta|\beta|},
\end{equation}
for all $x\in K$ and $\xi\in\bbR^n$. Thus, one can use these symbols to define pseudo-differential operators as
\[
\sigma(X,D)f(x)= \frac{1}{(2\pi)^n}\int_{\bbR^n} e^{ix\cdot \xi}\, \sigma(x,\xi) \widehat{f}(\xi)\, d\xi. 
\]
One of the most important results on the problem of bounding these operators was obtained by Fefferman \cite{Fefferman}. This was possible because of his previous work with Stein \cite{fstein}, where they introduced Hardy spaces in several variables and found their relation with the $\operatorname{BMO}$ space of John and Nirenberg \cite{JN}, which ended up being the suitable spaces to perform a proper interpolation. His result is as follows:

\begin{thm}[Fefferman \cite{Fefferman}]\label{feff}
Let $\sigma\in S_{1-a, \delta}^{-\beta}\left(\bbR^n\right)$ with $0 \leqq \delta<1-a<1$ and $\beta<n a / 2$.
    \begin{itemize}
        \item [a)]  Then $\sigma(X, D)$ is bounded on $L^p$ for
$$
\left|\frac{1}{p}-\frac{1}{2}\right| \leqq \gamma=\frac{\beta}{n}\left[\frac{n / 2+\lambda}{\beta+\lambda}\right], \lambda=\frac{n a / 2-\beta}{1-a} .
$$
\item [b)] If $|1 / p-1 / 2|>\gamma$, then the symbol
$$
\sigma(x, \xi)=\sigma_{\alpha \beta}(\xi)=\frac{e^{i|\xi|^\alpha}}{1+|\xi|^\beta} \in S_{1-a ,0}^{-\beta}
$$
provides an operator $\sigma_{\alpha \beta}(D)$ unbounded on $L^p$.
\item [c)] Let $\sigma \in S_{1-a \delta}^{-n a / 2}$, so that the critical $L^p$ space is $L^1$. Although $\sigma(x, D)$ is unbounded on $L^1$, it is bounded from the Hardy space $H^1$ to $L^1$. 
    \end{itemize}
\end{thm}
The objective of this article is to extend this result to a certain type of manifolds $M$ for relatively new classes of symbols constructed by Safarov in \cite{Safarov}. Of course Fefferman's result extends immediately to the classical H\"ormander classes $\Psi_{\rho, \delta, loc}^m(M)$ on manifolds since those are defined by means of local charts. However, one has to remember that there is not complete freedom on the choice of the numbers $\rho, \delta$, as in the case of $\bbR^n$, because those classes must satisfy the restriction $\rho>1/2$ due to the coordinate invariance requirement. When working with the Safarov classes $\Psi_{\rho, \delta}^{m}\left(\Omega^{\kappa}, \nabla, \tau\right)$ (see Section \ref{prelimi} for the definition) we are able to extend Fefferman's result for the case $\rho>1/3$ and even $\rho>0$, but the price to pay is some extra geometrical conditions. These classes are not defined using the local charts theory, which relies on local Fourier transform, but using oscillatory integrals given by global symbols. What we mean by global symbols is that the whole symbol $\sigma$ will be well-defined in the cotangent bundle $T^*M$, and not only the principal part as in the classical theory. The latter is achieved by using a well-defined global object on $M$ such as a linear connection $\nabla$. Thus, the extra geometric conditions we mentioned before would be conditions on the connection, which are being symmetric or flat, depending on the condition on $\rho$ (see Theorem \ref{main}). 

Remember that a connection is called \emph{symmetric} if its torsion tensor $T$ is zero, moreover if its curvature tensor $R$ is also zero the connection is called \emph{flat}.  These conditions were used by Safarov to obtain good behaviour of compositions, $L^2$-boundedness, existence of parametrices, etc., and more recently an approximate spectral projection \cite{mcK}. 
\begin{rem}
    The classes $\Psi_{\rho, \delta}^{m}\left(\Omega^{\kappa}, \nabla, \tau\right)$ will not depend on $\nabla$ if we stick to the condition $\rho>1/2$, and in this case they coincide with the classical H\"ormander classes $\Psi_{\rho, \delta, loc}^m(M)$. In other cases, the classes do depend on the choice of connection. Fortunately, for the case $\rho>1/3$ we always have a canonical symmetric connection, the famous Levi-Civita connection. 
\end{rem}
Let us mention that this is not the only-existing approach to defining symbol classes on manifolds utilizing connections, for instance one can find the works of Widom \cite{w1, w2}, Fulling-Kennedy \cite{fk}, Sharafutdinov \cite{sha,sha2}, etc. Moreover, the idea of working with global symbols appear in other contexts such as Lie groups $G$, where one has access to a global Fourier transform based on the representation theory of the group. For example, the theories developed by the second author and collaborators for the cases of compact Lie groups \cite{RuT} and nilpotent Lie groups \cite{FiR}; where in both cases the ``phase space" $T^*G$ is replaced by the direct product of the Lie group and its unitary dual, i.e. they considered symbols defined on $G\times \widehat{G}$. We want to highlight that $L^p-L^p$ and $L^p-L^q$ results are available in these calculi, so we mention some of the corresponding works \cite{lpgrad, lplqgrad, lpcomp}. 

Coming back to our problem, we are going to consider it on complete Riemannian manifolds $(M, g)$ where operators will be acting on $\kappa$-densities. We have to say that the main difficulty when proving an analogue of Theorem \ref{feff} in our setting is the fact that the pseudo-differential operators are not defined by means of a Fourier transform, so that we have to deal directly with the Schwartz kernels of the operators and adapt the techniques for them. We overcome this problem by using representations of such kernels which are strongly related with the local geometry of the manifold. In order to handle this local geometry in the estimation processes we assume the manifold is of bounded geometry (see e.g. \cite{chee}). 
\begin{defn}\label{boundedGeometry}
    We say that a complete Riemannian manifold $(M, g)$ is of \emph{bounded geometry} if
    \begin{enumerate}
        \item The injectivity radius of $M$, denoted by $r_0$, is strictly positive.
        \item In any geodesic ball $B_{r}(x)$ the metric tensor $g_{ij}$ in normal coordinates is bounded in the $C^\infty$ topology of $T_xM$, and its inverse tensor $g^{ij}$ is bounded in sup norm.  
    \end{enumerate}
\end{defn}
Some examples of manifolds satisfying all previous conditions are Euclidean spaces, compact manifolds, Lie groups, and some non-compact symmetric spaces. An important consequence for us of the bounded geometry assumption is that we know the growth behaviour of the balls of any radius in $M$ (not necessarily the ones with radius smaller than the injectivity radius), it is established in the following lemma by Taylor \cite{Taylor1} (this is a particular case of a more general result, Proposition 4.1 in \cite{chee}). 
\begin{lem}
\label{volGrowth}
    Let $(M,g)$ be a complete Riemannian manifold of bounded geometry. Then there exists $C_0=C_0(M)$, $\mu_0=\mu_0(M)$, and $K_0=K_0(M)$ such that for each $x\in M$, $r\in (0,\infty)$, we have
    \begin{equation}
        Vol B_r(x) \leq C_0(1+r)^{\mu_0}e^{K_0 r}. 
    \end{equation}
\end{lem}
Constants $C_0$, $\mu_0$, and $K_0$ are fixed in the rest of the article. 

As we mentioned before, our symbols classes will be defined by means of a linear connection $\nabla$, and in order to perform certain estimations we will need a similar condition as $(2)$ in the previous definition but for the Christoffel symbols of the connection:
\begin{defn}
    We say that a linear connection $\nabla$ is bounded if in any geodesic ball $B_{r}(x)$ the Christoffel symbols $\Gamma_{ij}^k$ in normal coordinates are bounded in the $C^\infty$ topology of $B_{r}(x)$.  
\end{defn}
\begin{rem}
    Note that if $(M,g)$ is of bounded geometry, then condition $(2)$ of Definition \ref{boundedGeometry} implies that the Levi-Civita connection $\nabla_{LC}$ is bounded. Therefore in this situation we can at least find a bounded connection. Henceforth when we ask the triple $(M,g,\nabla)$ to be of \emph{bounded geometry} we mean that $(M,g)$ is a complete Riemannian manifold of bounded geometry and that $\nabla$ is bounded. 
\end{rem} 

Moreover, one has to be a bit careful since we are allowing operators to act on $\kappa$-densities and not only on functions. Therefore one may adapt definitions of function spaces to this situation, but this is possible since the bundle of $\kappa$-densities is a trivial line bundle. 

We will say that an operator is \emph{locally bounded} on $L^p(M,\Omega^\kappa)$ if it is bounded from $L^p_{comp}(M,\Omega^\kappa)$ to $L^p_{loc}(M,\Omega^\kappa)$.  We will say it is \emph{globally bounded} if it is bounded from $L^p(M,\Omega^\kappa)$ to $L^p(M,\Omega^\kappa)$. 

Let us enunciate the main result of this article regarding the $L^p-L^p$ boundedness of operators belonging to Safarov classes (for precise definitions see Section \ref{prelimi}):
\begin{thm}
\label{main}
Let $(M,g, \nabla)$ be a complete Riemannian manifold of bounded geometry. Let $\theta\in [0,\frac{n(1-\rho)}{2})$ and $a\in S_{\rho, \delta}^{-\theta}(\nabla)$ with $0 \leq \delta<\rho\leq 1$. Suppose that at least one of the following conditions is fulfilled:
\begin{enumerate}
    \item $\rho>\frac{1}{2}$;
    \item the connection $\nabla$ is symmetric and $\rho>\frac{1}{3}$;
    \item the connection $\nabla$ is flat.
\end{enumerate}
Then we have the following.
 \begin{itemize}
        \item [a)]  The operator $a(X, D)$ is bounded from $L^p_{comp}(M, \Omega^\kappa)$ to $L^p_{loc}(M, \Omega^\kappa)$ for 
$$\left|\frac{1}{p}-\frac{1}{2}\right|\leq \frac{\theta}{n(1-\rho)}.$$  
\item [b)] This result is sharp, i.e., if $|1/p-1/2|>\theta/(n(1-\rho))$, then the symbol 
$$a(x,\xi) = a_{\rho\theta}(\xi) = \frac{e^{i|\xi|_{x_0}^{1-\rho}}}{1 + |\xi|_{x_0}^\theta}\in S_{\rho,0}^{-\theta}(\nabla)$$
gives an operator $a_{\alpha\beta}(D)$ unbounded on $L^p(\bbS^1, \Omega^\kappa)$, where $x_0$ is any fixed point on $\bbS^1$. 
\item[c)] Let $a\in S_{\rho, \delta}^{-\frac{n(1-\rho)}{2}}(\nabla)$, so that the critical $L^p(M, \Omega^\kappa)$ space is $L^1(M, \Omega^\kappa)$. Then $a(X, D)$ is bounded from the Hardy space $H_{comp}^1(M, \Omega^\kappa)$ to $L_{loc}^1(M,\Omega^\kappa)$. 
    \end{itemize}
    Moreover, if $a\in S_{\rho, \delta}^{-\theta}(\nabla)_{glo}$ then these boundedness results are global. 
\end{thm}
Furthermore, as an application of Theorem \ref{main} combined with some results on Bessel potentials (see Section \ref{prelimi}) we are able to obtain the following boundedness on Sobolev and Besov spaces, and some $L^p-L^q$ results: 
\begin{thm}\label{main2}
    Let $(M,g, \nabla)$ be a complete Riemannian manifold of bounded geometry. Let $\theta\in \bbR$, $0<q\leq \infty$, and $a\in S_{\rho, \delta}^{-\theta}(\nabla)$ with $0 \leq \delta<\rho\leq 1$. Let $s_1,s_2\in\bbR$ be such that
    \begin{equation}
        -\frac{n(1-\rho)}{2}<-\theta-s_1+s_2\leq 0.
    \end{equation}
    Suppose at least one of the following conditions is fulfilled:
\begin{enumerate}
    \item $\rho>\frac{1}{2}$;
    \item the connection $\nabla$ is symmetric and $\rho>\frac{1}{3}$;
    \item the connection $\nabla$ is flat.
\end{enumerate}
Then $a(X, D)$ is bounded from $L_{s_1, comp}^p(M,\Omega^\kappa)$ to $L_{s_2, loc}^p(M, \Omega^\kappa)$ for 
$$\left|\frac{1}{p}-\frac{1}{2}\right|\leq \frac{\theta}{n(1-\rho)}.$$
Moreover, if $\theta\in [0,\frac{n(1-\rho)}{2})$ and $s_1=s_2=s$, for $s\in \bbR$ fixed, then $a(X, D)$ is bounded from $B_{p,q}^s(M,\Omega^\kappa)$ to $B_{p,q}^s(M,\Omega^\kappa)$ for the same range of the values of $p$.  Furthermore, if $a\in S_{\rho, \delta}^{-\theta}(\nabla)_{glo}$, these previous boundedness results are global. 
\end{thm}
\begin{thm}
    Let $(M,g, \nabla)$ be a complete Riemannian manifold of bounded geometry. Let $\theta\in\bbR$ and $a\in S_{\rho, \delta}^{-\theta}(\nabla)$ with $0 \leq \delta<\rho\leq 1$. Suppose at least one of the following conditions is fulfilled:
\begin{enumerate}
    \item $\rho>\frac{1}{2}$;
    \item the connection $\nabla$ is symmetric and $\rho>\frac{1}{3}$;
    \item the connection $\nabla$ is flat.
\end{enumerate}
Then we have the following statements.
 \begin{itemize}
        \item [a)]  If $1<p\leq2\leq q<\infty$, then $a(X, D)$ is bounded from $L^p_{comp}(M, \Omega^\kappa)$ to $L^q_{loc}(M, \Omega^\kappa)$ if 
    \begin{equation*}\label{LPLQ1}
    n\left(\frac{1}{p}-\frac{1}{q}\right)\leq -\theta; 
\end{equation*} 
\item [b)] if $1<p\leq q \leq 2$ and 
        \begin{equation*}
        \label{condLpLqUltima}
           n\left(\frac{1}{p}-\frac{1}{q} + (1-\rho)\left(\frac{1}{q}-\frac{1}{2}\right)\right)\leq -\theta,
        \end{equation*}
        then $a(X, D)$ is bounded from $L^p_{comp}(M, \Omega^\kappa)$ to $L^q_{loc}(M, \Omega^\kappa)$;
         \item[c)] if $2\leq p\leq q <\infty $ and 
        \begin{equation*}
           n\left(\frac{1}{p}-\frac{1}{q} + (1-\rho)\left(\frac{1}{2}-\frac{1}{p}\right)\right)\leq -\theta, 
        \end{equation*}
        then $a(X, D)$ is bounded from $L^p_{comp}(M, \Omega^\kappa)$ to $L^q_{loc}(M, \Omega^\kappa)$. 
    \end{itemize}
    Moreover, if $a\in S_{\rho, \delta}^{-\theta}(\nabla)_{glo}$, these previous boundedness results are global. 
\end{thm}
\begin{rem}
    Notice that in none of the previous Theorems we are requiring the connection to be metric, even though we are working with a Riemannian manifold. 
\end{rem}
For discussions of our main results see Remarks \ref{rem1}, \ref{rem2} and \ref{rem3}.

Finally, let us comment on Conditions (1)-(3) of these theorems. Condition $(1)$ is recovering the classical results on local classes; remember that the restriction $\rho>1/2$ always appears as the intersection of $0\leq \delta<\rho\leq 1$ and $1-\rho \leq \delta$, since in this case the Safarov classes are independent of the choice of the connection $\nabla$ (resembling invariance under change of coordinates). Therefore, in this situation one can always pick connections that simplify the problems without worrying too much. Condition $(2)$ is extending the classical results since now we can take smaller $\rho>1/3$, allowing us to consider new types of operators, but this happens in some new classes since in this case the Safarov classes do depend on $\nabla$. Nevertheless, asking $\nabla$ to be symmetric (or torsion-free) is not a restriction at all on the geometry of $M$ because given $\nabla$ with Christoffel symbols $\Gamma_{k j}^{i}$, the connection given by the new Christoffel symbols $\tilde{\Gamma}_{k j}^{i} = \frac{1}{2}(\Gamma_{k j}^{i}+\Gamma_{jk}^{i})$ is always symmetric and it has the same geodesics (see e.g. \cite{kob}). Moreover, if one is interested on having metricity as well, for instance to write Christoffel symbols in terms of the Riemannian metric, one has always the Levi-Civita connection. Condition (3) is the one extending the most the types of operators we can consider because we allow $\rho>0$, but in this case we do have the major restriction of flatness; which together with the bounded geometry will reduce the manifolds to which our results are applicable. Regarding this conditions, if we have a flat metric connection on a compact Riemannian manifold, it is well-known that there are only finitely many of them in each dimension due to Bieberbach Theorems (see e.g. \cite{bib}) and basically all of them are tori, products among themselves and an already flat one quotient by free group actions. Interestingly, if one forgets about metricity i.e. we consider only a flat connection on a compact Riemannian manifold it is known that there exist infinitely many of them in every dimension \cite{aus}. Up to our knowledge, there is still no complete classification of this class of manifolds. 

To avoid confusion, we clarify that over the whole article we implicitly use standard Einstein notation. 

\section{Preliminaries}
\label{prelimi}
In this section we recall the Safarov pseudo-differential calculus and its main results; definitions of the BMO, Hardy, Besov and Sobolev spaces in complete Riemannian manifolds of bounded geometry together with the relevant properties we require for the purposes of this work. 
\subsection{Safarov calculus}
In this subsection we recall briefly the global pseudo-differen\-tial calculus on manifolds developed by Safarov in \cite{Safarov}, of course we refer to his paper for more details. We also prove some extra results we will need in Section \ref{S3}. 

Let $M$ be a $n$-dimensional smooth manifold. The linear operators we are going to consider will be acting on the space of $\kappa$-densities (sections of a certain line bundle) for $\kappa\in\bbR$, so we quickly define them and we refer to \cite{berl} for more details. The $\kappa$-density bundle is defined as the associated bundle of the following representation of $\operatorname{GL}(n)$: 
\begin{align*}
    \operatorname{GL}(n) &\xrightarrow{\rho} \operatorname{GL}(1)\\
    A & \mapsto |\operatorname{det}A|^{-\kappa},
\end{align*}
i.e. the bundle $\Omega^\kappa$ is defined as 
\begin{figure}[H]
\centering
\renewcommand\figurename{Diagram}
\begin{tikzpicture}[descr/.style={fill=white,inner sep=1.5pt}]
        \matrix (a) [
            matrix of math nodes,
            row sep=2.8em,
            column sep=2.5em,
            text height=1.5ex, text depth=0.25ex
        ]
        {
         \Omega^\kappa := \operatorname{Fr}(M) \times_{\rho} \bbR  \\
          M\\ 
        };
        \path[overlay,->, font=\scriptsize,>=latex]
        (a-1-1) edge node[midway,above] {$ $} (a-2-1);
\end{tikzpicture}    
\end{figure} 
\noindent where $\operatorname{Fr}(M)$ is the frame bundle of $M$, and as usual we denote the space of smooth sections as $\Gamma(\Omega^\kappa)$. Locally this means that we are dealing with objects (sections) that under a change of coordinates behave as follows
\[
u(y) = |\operatorname{det} (\partial x^i/\partial y^j)|^\kappa u(x(y)),
\]
where $\{x^k\}, \{y^k\}$ are two coordinates systems on $M$. For example, if $\kappa=0$, then $\Gamma(\Omega^0) = C^\infty(M)$ is the standard space of smooth functions on $M$; or if $M$ is oriented and $\kappa=1$, then $\Gamma(\Omega^1) = \Gamma(\bigwedge^n T^*M)$ is the space of $n$-forms on $M$. The density bundle ($\kappa=1$) is trivializable, so all of them are, therefore we fix a section $|dx|$ and we can use it to identify all these line bundles 
\begin{figure}[H]
\centering
\renewcommand\figurename{Diagram}
\begin{tikzpicture}[descr/.style={fill=white,inner sep=1.5pt}]
        \matrix (a) [
            matrix of math nodes,
            row sep=2.8em,
            column sep=2.5em,
            text height=1.5ex, text depth=0.25ex
        ]
        {
         \Omega & & \Omega^\kappa& \Omega^s  \\
          M  & & M & M\\
        };

        \path[overlay,->, font=\scriptsize,>=latex]
        (a-1-1) edge node[midway,above] {$ $} (a-2-1)
        (a-1-3) edge node[midway,above] {$ $} (a-2-3)
        (a-1-4) edge node[midway,above] {$ $} (a-2-4)
        (a-2-1) edge[out=120,in=240,black] node[xshift=-2ex] {$|dx|$} (a-1-1)
        (a-2-3) edge[out=120,in=240,black] node[xshift=-1ex] {$u$} (a-1-3)
        (a-2-4) edge[out=60,in=300,black] node[xshift=4ex] {$u|dx|^{s-\kappa}$} (a-1-4)
        (a-1-3) edge node[midway,above] {$\ast |dx|$} (a-1-4);
\end{tikzpicture}    
\end{figure}
via multiplication by $|dx|$. 
\begin{rem}
    Locally one may think of elements $u\in \Gamma(\Omega^\kappa)$ as 
\[
u(x) = f(x^1, \ldots, x^n) |dx^1\wedge \cdots \wedge dx^n|^\kappa.
\]
Moreover, there is a natural pairing between $\Omega^\kappa$ and $\Omega^{1-\kappa}$ given by integration 
\[
\Omega^\kappa\times \Omega^{1-\kappa} \xrightarrow{\langle \cdot, \cdot\rangle=\int}\mathbb{R}
\]
since our fixed section $|dx|$ is defining a measure on $M$ even if it is not oriented. 
\end{rem}
Now, the main idea to define the symbol classes over the cotangent bundle $T^*M$ is to choose a linear connection to split this bundle in two well defined spaces and to have a differential object to measure the horizontal growth of the symbols (there is always a natural way to do it on the vertical part). Using this construction we will have full symbols defined on $T^*M$ rather than just symbols with a well defined principal part. 
Remember that a connection $\nabla$, by definition, is a splitting of the following exact sequence 
\begin{figure}[H]
\centering
\renewcommand\figurename{Diagram}
\begin{tikzpicture}[descr/.style={fill=white,inner sep=1.5pt}]
        \matrix (a) [
            matrix of math nodes,
            row sep=2.8em,
            column sep=2.5em,
            text height=1.5ex, text depth=0.25ex
        ]
        {
         0 & VTM & TTM & HTM & 0,\\
        };

        \path[overlay,->, font=\scriptsize,>={Latex[length=2mm, width=2mm]}]
        (a-1-1) edge node[midway,above] {$ $} (a-1-2)
        (a-1-2) edge node[midway,left] {$ $} (a-1-3)
        (a-1-3) edge node[midway,right] {$ $} (a-1-4)
        (a-1-4) edge node[midway,right] {$ $} (a-1-5);
        \draw[->,loosely dashed,-{Latex[length=2mm, width=2mm]}] (a-1-4) edge[out=120,in=50,black] node[yshift= 1.2ex] {$\nabla$} (a-1-3);
\end{tikzpicture}   
\end{figure}
\noindent and abusing notation we are going to denote by $\nabla$ also the dual connection induced on $T^*M$. The existence of this splitting simply means that we choose a horizontal space $HT^*M$ such that $TT^*M\cong VT^*M\oplus HT^*M$, where $VT^*M$ is the natural vertical part which, if we denote points on $T^*M$ as $(y,\eta)$, is generated by vector fields $\{\partial_{\eta_j}\}\subset \mathfrak{X}(T^*M)$. On the other hand, the connection $\nabla$ will give us the generators of the horizontal part by considering horizontal lifts of vector fields $\partial_{y^k}\in \mathfrak{X}(M)$. Let $\Gamma_{k j}^{i}$ be the Christoffel symbols associated to $\nabla$, given a vector field $v=\sum v^{k}(y) \partial_{y^{k}}$ on $M$, its horizontal lift is defined as 
$$
\nabla_{v}=\sum_{k} v^{k}(y) \partial_{y^{k}}+\sum_{i, j, k} \Gamma_{k j}^{i}(y) v^{k}(y) \eta_{i} \partial_{\eta_{j}}. 
$$
In particular, if $v=\partial_{y^{k}}$ we will denote its horizontal lift by $\nabla_k$, these are the derivatives that will appear in the definition of the symbols classes. To define the phases and the kernels of the pseudo-differential operators we will need some other geometric objects related to $\nabla$, so we quickly fix the notation:  
\begin{itemize}
\item We usually use normal coordinates systems so we use the abbreviation n.c.s.
\item We denote $T=\{T_{jk}^i\}$ the torsion tensor, $R=\{R_{jkl}^i\}$ the curvature tensor and $\{R_{kl}\}=\{\sum_i R_{kil}^i\}$ the Ricci tensor. If $T\equiv 0$ we say that $\nabla$ is symmetric, and if in addition $R\equiv 0$ we say that $\nabla$ is flat.
\item Given a neighborhood $U_x$ of $x$, we denote by $\gamma_{y,x}(t)$ the shortest geodesic joining $x$ and $y\in U_x$. We will use the notation $z_t= \gamma_{y,x}(t)$.
\item We denote $\Phi_{y,x}:T_x^*M\to T_y^*M$ the parallel transport along $\gamma_{y,x}(t)$ and we define $\Upsilon_{y,x} := \Upsilon_{x}(y)=|\operatorname{det} \Phi_{y,x}|$, the latter function is a $1$-density in $x$ and a $(-1)$-density in $y$. 
\item The following useful formulas, in normal coordinates, were proved by Safarov in \cite{Safarov}: 
\begin{align}
     \partial_{y^k}(\Phi_{y,x})_j^i\big|_{y=x}&=\Gamma_{kj}^i(y)\big|_{y=x}=\frac{1}{2}T_{kj}^i(x),\label{trasp}\\
     \partial_{y^k}(\Upsilon_{x}(y))\big|_{y=x}&=\sum_j\Gamma_{kj}^j(y)\big|_{y=x}=\frac{1}{2}\sum_jT_{kj}^j(x).\label{determiTrans}
\end{align}
    \end{itemize}
\begin{rem}
   One has to be careful when taking horizontal and vertical derivatives because they will produce new tensors (since one may think of coordinates $\eta$ of being $1$-forms); for example, if $a\in C^\infty(T^*M)$ then $\{\partial_\eta^\alpha a\}_{|\alpha|=p}$ would be a $(p,0)$-tensor and $\{\nabla_{i_{1}} \ldots \nabla_{i_{q}} a\}$ would be a $(0,q)$-tensor. Actually, for the horizontal ones one can define another derivative using parallel transport 
   \[
   \nabla_x^\alpha a(x,\eta) = \frac{d^\alpha}{dy^\alpha} a(y, \Phi_{y,x}\eta)\bigg|_{y=x},
   \]
   which will coincide (\cite[Corollary 2.5]{Safarov}) and often will appear in formulas below. 
\end{rem}

\begin{defn}\label{defSym}
     Let $0\leq \delta<\rho \leq 1$, the space $S_{\rho, \delta}^m(\nabla)$ denotes the class of functions $a\in C^\infty(T^*M)$ such that in any coordinates $y$, for all $\alpha$ and $i_1, \ldots, i_q$, 
$$
\left|\partial_{\eta}^{\alpha} \nabla_{i_{1}} \ldots \nabla_{i_{q}} a(y, \eta)\right| \leq \mathrm{const}_{K, \alpha, i_{1}, \ldots, i_{q}}\langle\eta\rangle_{y}^{m+\delta q-\rho|\alpha|},
$$
where $y$ runs over a compact set $K\subset M$, $\langle\eta\rangle_{y}:= (1 + w^2(y, \eta))^{1/2}$ where $w\in C^\infty(T^*M\setminus 0)$ is a positive function homogeneous in $\eta$ of degree 1. Also amplitudes are defined as satisfying 
$$
\left|\partial_{z}^{\beta} \partial_{\eta}^{\alpha} \nabla_{i_{1}} \ldots \nabla_{i_{q}} a(z ; y, \eta)\right| \leq \mathrm{const}_{K, \alpha, \beta, i_{1}, \ldots, i_{q}}\langle\eta\rangle_{y}^{m+\delta|\beta|+\delta q-\rho|\alpha|}, \text{ for all } y,z\in K.
$$
\end{defn}
\begin{rem}
    According to Safarov this symbols classes are independent of the function $w$, but for our purposes we will work with a Riemannian manifold $(M,g)$, so will stick to $w(y,\eta) = |\eta|_y = \sqrt{g^{ab}(y)\eta_a\eta_b}$, where $g^{ab}$ are the dual coefficients of the metric $g$. 
\end{rem}
Of course, symbols in $S_{\rho, \delta}^m(\nabla)$ satisfy the classical properties concerning multiplication, addition, and differentiation with respect of each of the variables.
\begin{prop}
We have the following properties.
     \begin{itemize}
            \item If $a\in S_{\rho, \delta}^{m_1}(\nabla)$ and $b\in S_{\rho, \delta}^{m_2}(\nabla)$, set $m=\operatorname{max}(m_1, m_2)$, then 
            $$ab\in S_{\rho, \delta}^{m_1+m_2}(\nabla) \text{ and } a+b\in S_{\rho, \delta}^{m}(\nabla).$$
            \item Moreover,
            $\partial_\eta^\alpha a \in S_{\rho, \delta}^{m-\rho|\alpha|}(\nabla)$, $\nabla_{v_1}\cdots\nabla_{v_q}a\in S_{\rho, \delta}^{m+\delta q}(\nabla)$
            for all $a\in S_{\rho, \delta}^{m}(\nabla)$ and any $v_1, \ldots, v_q\in \mathfrak{X}(M).$
            \end{itemize}
\end{prop}
As usual, for any $l\in\bbN$ we define the seminorms for a compact set $K\subset M$
$$\|a\|_{l; S_{\rho, \delta}^{m}}:= \sup_{|\alpha|+ q \leq l,\, (y,\eta) \in K\times T_y^*M} \frac{\left|\partial_{\eta}^{\alpha} \nabla_{i_{1}} \ldots \nabla_{i_{q}} a(y, \eta)\right|}{\langle\eta\rangle_{y}^{m+\delta q-\rho|\alpha|}},$$
which regularly will appear in the estimation process of Section \ref{ultima}. 

\begin{ex}
\label{cotaDerivadademetrica}
    To give an example on how to manipulate the seminorms, which actually will be used in Section \ref{ultima}, let us suppose for a moment that $(M,g,\nabla)$ is a complete Riemannian manifold of bounded geometry. We investigate the seminorm of the powers of the weight $\langle\eta\rangle_{y}$, i.e. let $\gamma\in\bbC$ and $l\in\bbN$, and we want to estimate $\|\langle\eta\rangle_{y}^\gamma\|_{l; S_{\rho, \delta}^{m}}$. On one hand, using symmetry of the metric tensor
    \begin{align*}
        \nabla_k\langle\eta\rangle_{y}^\gamma &= \nabla_k (1 + g^{ab}(y)\eta_a\eta_b)^{\gamma/2}\\
        &= \partial_k(1 + g^{ab}(y)\eta_a\eta_b)^{\gamma/2} + \Gamma_{kj}^i\eta_i\partial_{\eta_j}(1 + g^{ab}(y)\eta_a\eta_b)^{\gamma/2}\\
        &= \frac{\gamma}{2}(1 + g^{ab}(y)\eta_a\eta_b)^{(\gamma-2)/2}(\partial_k g^{ab})\eta_a\eta_b + \gamma \Gamma_{kj}^i\eta_i(1 + g^{kb}(y)\eta_k\eta_b)^{(\gamma-2)/2}g^{kb}\eta_b\\
        &\leq \gamma C (1 + g^{ab}(y)\eta_a\eta_b)^{\gamma/2} = C\gamma \langle\eta\rangle_{y}^\gamma. 
    \end{align*}
    So we see that $\nabla_{i_{1}} \ldots \nabla_{i_{q}} \langle\eta\rangle_{y}^\gamma$ would behave as a polynomial in $\eta$ of degree $\gamma$ having coefficients depending on partial derivatives of the metric coefficients and Christoffel symbols, which are bounded since $(M,g,\nabla)$ is of bounded geometry, and powers of $\gamma$ depending on $q$. Therefore 
    \[
    |\nabla_{i_{1}} \ldots \nabla_{i_{q}} \langle\eta\rangle_{y}^\gamma| \leq |p_{1,q}(\gamma)| \langle\eta\rangle_{y}^\gamma,
    \]
    where $p_{1,q}(\gamma)$ is a polynomial of degree $q$. On the other hand, 
    \begin{align*}
        \partial_{\eta_k} \langle\eta\rangle_{y}^\gamma &= \partial_{\eta_k} (1 + g^{ab}(y)\eta_a\eta_b)^{\gamma/2}\\
        &= \gamma (1 + g^{kb}(y)\eta_k\eta_b)^{(\gamma-2)/2}g^{kb}\eta_b\\
        &\leq \gamma C \langle\eta\rangle_{y}^{\gamma-1}.
    \end{align*}
    Therefore, 
    \[
    |\partial_{\eta}^{\alpha} \langle\eta\rangle_{y}^\gamma| \leq |p_{2,\alpha}(\gamma)| \langle\eta\rangle_{y}^{\gamma-|\alpha|},
    \]
    where $p_{2,\alpha}(\gamma)$ is a polynomial of degree $|\alpha|$. Thus, combining the two previous calculations we find that for $|\alpha|+ q \leq l$, we have
    \[
    \left|\partial_{\eta}^{\alpha} \nabla_{i_{1}} \ldots \nabla_{i_{q}} \langle\eta\rangle_{y}^\gamma\right| \leq |p_l(\gamma)| \langle\eta\rangle_{y}^{\gamma-|\alpha|},
    \]
    where $p_l(\gamma)$ is a polynomial of order $l$. Finally, we conclude that
    \[
    \|\langle\eta\rangle_{y}^\gamma\|_{l; S_{\rho, \delta}^{m}}\leq \sup_{|\alpha|+ q \leq l,\, (y,\eta) \in K\times T_y^*M} |p_l(\gamma)|\frac{1}{\langle\eta\rangle_{y}^{(m-\gamma+|\alpha|)+\delta q-\rho|\alpha|}}.
    \]
\end{ex}
Before giving the definition of a pseudo-differential operator, we still need to define what would be the phase function of these operators. 
\begin{defn}
    Let $V$ be a sufficiently small neighborhood of the diagonal $\Delta$ in $M\times M$. We introduce the phase functions 
$$\varphi_\tau(x,\zeta,y)= -\langle \dot{\gamma}_{y,x}(\tau), \zeta\rangle, \text{ where }(x,y)\in V,\quad \tau\in [0,1],\quad  \zeta\in T^*_{z_\tau}M.$$
\end{defn}
\begin{defn}
    Let $A: \Gamma_c(\Omega^\kappa)\to \Gamma(\Omega^\kappa)$ be a linear operator with Schwartz kernel $\mathscr{A}(x, y)$, i.e, $\langle Au, v\rangle = \langle\mathscr{A} , u\otimes v\rangle$. We say that $A$ is pseudo-differential if
      \begin{enumerate}
          \item $\mathscr{A}(x, y)$ is smooth in $(M\times M)\setminus \Delta$.
          \item On a neighborhood $V$ of $\Delta$ the Schwartz kernel is represented by an oscillatory integral of the form
          \begin{align*}
              \mathscr{A}(x, y)=&\frac{1}{(2\pi)^n}p_{\kappa, \tau} \int_{T^*_{z_\tau}M} e^{i \varphi_{\tau}(x, \zeta, y)} a\left(z_{\tau}, \zeta\right)\,d \zeta,\text { for }(x, y) \in V,
          \end{align*}
          where $a\in S_{\rho, \delta}^m(\nabla)$ and $p_{\kappa, \tau} = p_{\kappa, \tau}(x,y) = \Upsilon_{y, z_\tau}^{1-\kappa}\Upsilon_{z_\tau, x}^{-\kappa}$.
      \end{enumerate}
We denote $\Psi_{\rho, \delta}^m\left(\Omega^\kappa, \nabla, \tau\right)$ these classes of pseudo-differential operators. Moreover, we say that a pseudo-differential operator $A$ is properly supported if both projections $\pi_1,\pi_2: \operatorname{supp}\mathscr{A}\to M$ are proper maps. 
\end{defn}
\begin{rem}\label{kernels}
     Let $\chi$ be a smooth function being equal to $1$ in $V$ and vanishing outside. We define the smooth function 
     \begin{equation}\label{kernelDef}
         K(x,y):= (1-\chi(x,y)) \mathscr{A}(x, y) \in C^\infty(M\times M),
     \end{equation}
     so one can write a pseudo-differential operator $A$ as follows:
     \begin{align*}
         Au(x) &= A_{loc}u(x) + A_{glo}u(x)\\
         &=\int_M \mathscr{A}(x, y) u(y) \chi(x,y) |dy| + \int_M K(x,y)u(y) |dy|\\
         &=\frac{1}{(2\pi)^n}\int_M \int_{T_{z_\tau}^*M}  p_{\kappa, \tau}e^{i \varphi_{\tau}(x, \zeta, y)} a\left(z_{\tau}, \zeta\right) \chi(x,y) u(y)d \zeta |dy|+ \int_M K(x,y)u(y) |dy|.
     \end{align*}
     Notice that our definition of pseudo-differential operator only allow us to write explicitly the local part of the operator, $A_{loc}$, so that estimations in Section \ref{ultima} would be local unless one imposes a decay condition on the global part $A_{glo}$ of the operator. Classically on $\bbR^n$ one has polynomial decay on the kernel for $x\neq y$ coming directly from the conditions on the symbol classes since in that case the conditions are global (see e.g. \cite[Theorem 2.3.1]{RuT}). In the setting of complete Riemannian manifolds of bounded geometry such properties would not be enough in general because due to Lemma \ref{volGrowth} one could have exponential growth of the balls and this behaviour needs to be compensated by the kernel. For this reason, in the next sections we are going to use the following assumption on the kernel $K$ away from the diagonal $\Delta$. 
\end{rem}
\begin{as}\label{assu}
    Let $(M,g)$ be a complete Riemannian manifold of bounded geometry of dimension $n$, and denote by $d$ the associated Riemannian distance. Let $C_0$, $\mu_0$, and $K_0$ be as in Lemma \ref{volGrowth}. We say that $a\in S_{\rho,\delta}^m(\nabla)_{glo}$ if the corresponding kernel $K$ of the pseudo-differential operator $a(X,D)$ in \eqref{kernelDef} away of the diagonal satisfies: 
    \begin{itemize}
        \item If $K_0>0$, then there exist $W>K_0$ such that 
        \[
        \left|K(x,y)\right|\leq C e^{-W d(x,y)}, \text{ for all } x,y\in M \text{ such that } d(x,y)\geq 1. 
        \]
        \item If $K_0=0$, then there exist $L>\mu_0+n$ such that 
        \[
        \left|K(x,y)\right|\leq C d(x,y)^{-L}, \text{ for all } x,y\in M \text{ such that } d(x,y)\geq 1.
        \]
    \end{itemize} 
\end{as}
 As we already mentioned, decay as in the second item is satisfied by all pseudo-differential operators belonging to H\"ormander classes. Notice that on a compact Riemannian manifold such properties are trivially satisfied since away from the diagonal the kernel is a smooth function. On the other hand, some operators satisfying the first item are discussed in detail by Cheeger, Gromov and Taylor in \cite{chee}. Specifically, they treated functions of the Laplacian $f(\sqrt{-L})$ where $f\in S_\rho^m(\bbR)$ (see definition of this class in Example \ref{ex1}) and $f$ is holomorphic on the strip 
    \[
    \Omega_W = \{z\in\bbC : |\operatorname{Im}z|<W\}, 
    \]
    where the width of the strip is precisely the constant appearing in the exponential decay of the kernels.
    
We move on, and now let us describe locally the kernel $\mathscr{A}(x, y)$ corresponding to the operator $A_{loc}$. For that purpose take a chart $U\times U\subset V$, and  let $\{x^k\}, \{y^k\}$ be the same coordinates. Then there exist a non-degenerate $n\times n$ matrix $\Psi_\tau(x,y)$ such that for all $(x,y)\in U\times U$
\[
\varphi_{\tau}(x, \zeta, y) = (x-y)\cdot \Psi_\tau \zeta,
\]
and changing coordinates $\tilde{\zeta} = \Psi_\tau \zeta$ we obtain 
\begin{equation}
\label{changeOfCoor}
 \mathscr{A}(x, y) = \frac{1}{(2\pi)^n}p_{\kappa, \tau} |\operatorname{det}\Psi_\tau|^{-1}\int_{T^*_{z_\tau}U} e^{i(x-y)\cdot\Psi_\tau^{-1}\tilde{\zeta}} a\left(z_{\tau}, \Psi_\tau^{-1}\tilde{\zeta}\right)\,d \tilde{\zeta},\text { for }(x, y) \in U\times U.    
\end{equation}
Now we state some of the results proved by Safarov \cite{Safarov} for this calculus that are going to be useful for us. 
\begin{prop}
\label{propTaoZero}
Let $0\leq \delta<\rho \leq 1$. For all $\tau, s \in[0,1]$ the classes $\Psi_{\rho, \delta}^m\left(\Omega^\kappa, \nabla, \tau\right)$ and $\Psi_{\rho, \delta}^m\left(\Omega^\kappa, \nabla, s\right)$ coincide, and
$$
\sigma_{A, s}(x, \xi) \sim \sum_\alpha \frac{1}{\alpha !}(\tau-s)^{|\alpha|} D_{\xi}^\alpha \nabla_x^\alpha \sigma_{A, \tau}(x, \xi) \quad \text { as }\langle\xi\rangle_x \rightarrow \infty.
$$
\end{prop}
\begin{rem}
\label{remCoNo}
    Proposition \ref{propTaoZero} allows us to focus on the case $\tau=0$ since all classes are equivalent when regarding the parameter $\tau$. Moreover, according to \eqref{changeOfCoor} we have in the normal coordinate system (n.c.s) $\{y^k\}$ centered at $x$ that locally a pseudo-differential operator $A$ is
    $$Au(x)= \frac{1}{(2\pi)^n}\int_{\bbR^n} \int_{T_x^*M}  \Upsilon_x^{1-\kappa}(y)e^{i(x-y)\cdot \zeta} a\left(x, \zeta\right) u(y) d\zeta |dy|.$$
\end{rem}
As in the classical calculus, we will need a tool that allows us to pass from amplitudes to symbols:
\begin{prop}
\label{AmpliToSym}
Let $0\leq \delta<\rho \leq 1$. Let $\tau, s \in[0,1]$ and let a be an amplitude from $\mathrm{S}_{\rho, \delta}^m(\nabla)$. Then the oscillatory integral
$$
\mathscr{A}(x, y)=\frac{1}{(2\pi)^n}p_{\kappa, \tau} \int_{T^*_{z_\tau}M} e^{i \varphi_\tau(x, \zeta, y)} a\left(z_s ; z_\tau, \zeta\right) d \zeta
$$
coincides with the Schwartz kernel of a pseudo-differential operator $A\in \Psi_{\rho, \delta}^m\left(\Omega^\kappa, \nabla, \tau\right)$ such that
$$
\begin{gathered}
\left.\sigma_{A, \tau}(x, \xi) \sim \sum_\alpha \frac{1}{\alpha !}(s-\tau)^{|\alpha|} D_{\xi}^\alpha \nabla_y^\alpha a(y ; x, \xi)\right|_{y=x} \quad \text { as }\langle\xi\rangle_x \rightarrow \infty, \\
\left.\sigma_{A, s}(x, \xi) \sim \sum_\alpha \frac{1}{\alpha !}(\tau-s)^{|\alpha|} D_\eta^\alpha \nabla_y^\alpha a(x ; y, \eta)\right|_{(y, \eta)=(x, \xi)} \quad \text { as }\langle\xi\rangle_x \rightarrow \infty.
\end{gathered}
$$
\end{prop}
Here below come the most important results because they are the ones that extend the classical results, allowing us to deal with new classes of operators thanks to the consideration of the connection $\nabla$ and the geometry that it implies. 
\begin{thm}
\label{adjunto}
     Let $0\leq \delta<\rho \leq 1$. If $A \in \Psi_{\rho, \delta}^{m}\left(\Omega^{\kappa}, \nabla\right)$ then $A^*\in \Psi_{\rho, \delta}^{m}\left(\Omega^{1-\kappa}, \nabla\right)$ and
    $$\sigma_{A^*, \tau}(x,\xi) \sim \sum_{\alpha}\frac{1}{\alpha!}(1-2\tau)^{|\alpha|}D_\xi^\alpha\nabla_x^\alpha \overline{\sigma_{A,\tau}(x,\xi)},$$
    
    as $\langle\xi\rangle_x\to\infty.$ In particular, $\sigma_{A^*, 1/2}(x,\xi)= \overline{\sigma_{A, 1/2}(x,\xi)}$.
\end{thm}
\begin{thm}{}
\label{producto}
Let $0\leq \delta<\rho \leq 1$. Let $A \in \Psi_{\rho, \delta}^{m_{1}}\left(\Omega^{\kappa}, \nabla\right), B \in \Psi_{\rho, \delta}^{m_{2}}\left(\Omega^{\kappa}, \nabla\right)$, and let at least one of these pseudo-differential operators be properly supported. Assume that at least one of the following conditions is fulfilled:
\begin{enumerate}
    \item $\rho>\frac{1}{2}$;
    \item the connection $\nabla$ is symmetric and $\rho>\frac{1}{3}$;
    \item the connection $\nabla$ is flat.
\end{enumerate}
Then $A B \in \Psi_{\rho, \delta}^{m_{1}+m_{2}}\left(\Omega^{\kappa}, \nabla\right)$ and
\begin{equation}
\label{asyExp}
    \sigma_{A B}(x, \xi) \sim \sum_{\alpha, \beta, \gamma} \frac{1}{\alpha !} \frac{1}{\beta !} \frac{1}{\gamma !} P_{\beta, \gamma}^{(\kappa)}(x, \xi) D_{\xi}^{\alpha+\beta} \sigma_{A}(x, \xi) D_{\xi}^{\gamma} \nabla_{x}^{\alpha} \sigma_{B}(x, \xi)
\end{equation}
as $\langle\xi\rangle_{x} \rightarrow \infty$. 
\end{thm} 
\begin{rem}
    Here the condition of being properly supported is guaranteeing that the product of the operators is well-defined. 
\end{rem}
\begin{rem}
    The functions $P_{\beta, \gamma}^{(\kappa)}(x, \xi)$ are the responsible of allowing us to lower down the order of $\rho$, since they are supporting extra decay because they are polynomial in $\xi$ with coefficients depending on the curvature, torsion and its covariant derivatives. For completeness we remind its definition: 
    \[
    P_{\beta, \gamma}^{(\kappa)}(x, \xi) = \left((\partial_y+\partial_z)^\beta \partial_y^\gamma \sum_{|\beta'|\leq |\beta|} \frac{1}{\beta'!}D_\xi^{\beta'}\partial_y^{\beta'}(e^{i\psi}\Upsilon_{\kappa})\right)\bigg|_{y=z=x}, 
    \]
    where 
    \begin{align*}
        &\Upsilon_{\kappa}(x,y,z)=\Upsilon_{y,z}^{1-\kappa}\Upsilon_{z,x}^{2-\kappa}\Upsilon_{x,y}^{1-\kappa}, \\
        \psi(x,\xi;y,&z)= \langle\dot{\gamma}_{y,x},\xi\rangle-\langle\dot{\gamma}_{z,x},\xi\rangle-\langle\dot{\gamma}_{y,z},\Phi_{z,x}\xi\rangle,
    \end{align*}
    in normal coordinates. 
\end{rem}
\begin{thm}
\label{L2bound}
Let $0\leq \delta<\rho \leq 1$. Let at least one of Conditions (1)-(3) of Theorem \ref{producto} be fulfilled. Then $A\in \Psi_{\rho, \delta}^{m}\left(\Omega^{\kappa}, \nabla\right)$ is bounded from $H^s_{\operatorname{comp}}(M, \Omega^\kappa)$ to $H^{s-m}_{\operatorname{loc}}(M, \Omega^\kappa)$ for all $s\in \bbR$.
\end{thm}
The latter results would be the key pieces to our proof of the $L^p$-boundedness. 
\subsubsection{Some examples of symbols}\label{examples}
We remind some examples of symbols discussed by Safarov and others, and we include some comments on the case of the torus $\bbT^n$. 

\begin{ex}[The Laplacian]\label{ex1}
    Let  $(M,g)$ be a Riemannian manifold, recall that the Laplace-Beltrami operator associated to the metric $g$ is given by 
\begin{equation}\label{laplaciano}
    Lu(x) = \frac{1}{\sqrt{|\operatorname{det}g_{ij}(x)|}}\partial_{x^i}\left(\sqrt{|\operatorname{det}g_{ij}(x)|} g^{ij}(x)\partial_{x^j}u(x)\right),
\end{equation}
where $g_{ij}$ and $g^{ij}$ are the metric and inverse metric coefficients, respectively. If one takes the Levi-Civita connection $\nabla_{LC}$ of $M$, then one can prove that the $\tau$-symbol of $L$ for any value of $\kappa$ and $\tau$ is given by 
\[
a_{L,\tau}(x,\xi)=-|\xi|_x^2+\frac{1}{3}S(x),
\]
where $S(x)$ is the scalar curvature of $M$. If one does not use $\nabla_{LC}$, but another connection $\nabla$ more terms related to the Christoffel symbols will appear in the previous formula.  This is already a relevant difference compared with the classical theory since in that case the principal symbol of the Laplacian is only given by $-|\xi|_x^2$. 

Moreover, let us also assume that $M$ is compact and let $H$ to be a first-order differential operator or a classical pseudo-differential operator. Consider the operator 
\[
P=(-L+H)^{1/2},
\]
and let $S_\rho^m(\bbR)$ denote the class of functions $\omega$ such that
\[
\left|\frac{d^k}{ds^k}\omega(s)\right|\leq C_k(1+|s|)^{m-k\rho} \text{ with } 0<\rho\leq 1.
\]
Therefore \cite[Theorem 11.2]{Safarov} guarantees that if $\omega\in S_\rho^m(\bbR)$, then $\omega(P)\in \Psi_{\rho,0}^m(\Omega^\kappa, \nabla_{LC})$ and we have the following asymptotic expansion for its $0$-symbol:
\[
a_{\omega(P)}(x,\xi)\sim \omega(|\xi|_x) + \sum_{j=1}^\infty c_j(x,\xi) \omega^{(j)}(|\xi|_x) \text{ as } |\xi|_x\to\infty, 
\]
where $c_j$ are functions determined recursively and $w^{(j)}(|\xi|_x)$ is the $j$th derivative of $\omega$ evaluated at $|\xi|_x$. 
\end{ex}
\begin{ex}[Connections arising from vector fields]\label{ex2}
    In this case $M$ is not necessarily a Riemannian manifold neither compact. Let $x\in M$ and let $\{\nu_1(x), \cdots,\nu_n(x)\}$ be a frame, i.e., a basis of $T_xM$. Then there exist functions $C_{jk}^i\in C^\infty(M)$ such that the commutator of this basis elements is written as 
    \[
    [\nu_j(x),\nu_k(x)] = C_{jk}^i(x)\nu_i(x). 
    \]
    From these functions one can construct a connection $\nabla$ such that its Christoffel symbols are given by
    \[
    \Gamma_{jk}^i(x) = -\tilde{\nu}_k^l(x)\partial_{x^j}\nu_l^i(x),
    \]
    where $\{\tilde{\nu}_1(x), \cdots,\tilde{\nu}_n(x)\}$ is the dual frame. Thus, if we denote by $A_l^{(\kappa)}$ the Lie differentiations along $v_l$ in the space of $\kappa$-densities and define 
    \[
    A_{(\kappa)}^\alpha(y,D_y)=\sum_{i_1, \cdots, i_q} A_{i_1}^{(\kappa)} \cdots A_{i_q}^{(\kappa)} \text{ with } q=|\alpha|,
    \]
    one can prove that assuming 
    \[
    C_{jk}^i\equiv 0 \text{ for all } i\geq k,
    \]
    that the symbol of such operators is given simply by 
    \[
    \sigma_{A_{(\kappa)}^\alpha}(x,\xi)= i^{|\alpha|}\langle v_1(x),\xi\rangle^{\alpha_1}\cdots\langle v_n(x),\xi\rangle^{\alpha_n},
    \]
    where $\langle\cdot,\cdot\rangle$ here means the natural pairing between vectors and covectors. 
    
    Notice that the previous construction is of particular interest when applied to parallelizable manifolds because in that case one has a global frame $\{\nu_1, \cdots,\nu_n\}\subset\mathfrak{X}(M)$, i.e., the tangent bundle is trivial $TM\cong M\times\bbR^n$; so it would be suitable to the case of Lie groups. 
\end{ex}

\begin{rem}
    The previous example was generalized by Shargorodsky in \cite{Shargorodsky} in two different ways simultaneously. On one hand, defining operators not only acting on functions or densities, but on vector bundles. On the other hand, constructing an anisotropic version of it. By this we mean that when considering homogeneous symbols, now one allows them to have different weights for the different directions corresponding to each vector field. This let the author to consider semi-elliptic operators instead of the classical elliptic ones. For this reason, he defined a particular function $w$ in Definition \ref{defSym} to measure properly each weight. Let us briefly introduce his symbol classes. Let $\bm{d}= (d_1, \cdots, d_n)\in \bbR^n$ be a vector such that for all $k$,  $d_k>0$ and 
    \[
    \sum_{k=1}^n\frac{1}{d_k}=n.
    \]
    For any multi-index $\alpha\in \bbN^n$ we put 
    \[
    |\alpha:\bm{d}|= \sum_{k=1}^n\frac{\alpha_k}{d_k}. 
    \]
    For any vector $\eta\in \bbR^n\backslash \{0\}$ let $w(x,\eta):= \tau(\eta)=|\eta|_{\bm{d}}$ be the unique solution to the equation
    \[
    \sum_{k=1}^n \tau^{-2/d_k}\eta_k^2=1. 
    \]
    So, let $m\in\bbR$, $0\leq\delta,\rho\leq 1$. The classes $S_{\rho,\delta}^{m,\bm{d}}(M\times \bbR^n)$ are functions $a\in C^\infty(M\times \bbR^n)$ such that in any coordinates $y$ running in any compact set $K\subset M$, for all $\alpha$ and $j_1, \ldots, j_q$,
    $$
\left|\partial_{\eta}^{\alpha} \partial_{\nu_{j_1}} \ldots \partial_{\nu_{j_q}} a(y, \eta)\right| \leq \mathrm{const}_{K, \alpha, j_{1}, \ldots, j_{q}}(1+|\eta|_{\bm{d}})^{m+\delta |\beta:\bm{d}|-\rho|\alpha:\bm{d}|},
$$
where $\beta$ is the multi-index corresponding to the set of indices $\{j_1, \cdots,j_q\}$. Shargorodsky obtained several results for these classes: adjoints, compositions, $L^p$-estimates, compactness, Fredholmness, etc. We point out that his $L^p$-results concern $0$-order pseudo-differential operators as in the classical result \cite[Chapter XI, Theorem 2.1]{Taylor3}, which is different from the Fefferman-type result of this article. 
\end{rem}

\begin{ex}[The torus $\bbT^n$]\label{ex4}
    Let us deal in a little more detail with the case of the compact flat metric Riemannian manifold par excellence, the torus $\bbT^n$. It is well known that the projection map 
    \[
    \pi: \bbR^n\to \bbR^n/\bbZ^n \cong \bbT^n
    \]
    provides the torus a canonical (flat) metric $g_{ij}=\delta_{ij}$. Thus, immediately we have that the Levi-Civita connection $\nabla_{LC}$ in this case is flat (and symmetric) due to the well-known formula for the Christoffel symbols of any Levi-Civita connection:
    \[
    \Gamma_{jk}^l = g^{lr}(\partial_k g_{rj} + \partial_j g_{rk} - \partial_r g_{jk}). 
    \]
    Thus, we are in a position to consider classes $\Psi_{\rho, \delta}^{m}\left(\Omega^{\kappa}, \nabla_{LC}, \tau\right)$ for  $0 \leq \delta<\rho\leq 1$. Since the cotangent bundle $T^*\bbT^n$ is trivial, symbols will be just functions $a\in C^\infty(\bbT^n \times \bbR^n)$ satisfying
    \[
    \left|\partial_{\eta}^{\alpha} \partial_y^\beta a(y, \eta)\right| \leq \mathrm{const}_{K, \alpha, \beta}(1+|\eta|^2)^{(m+\delta |\beta|-\rho|\alpha|)/2},
    \]
    which resembles very much Definition \ref{defrn} on $\bbR^n$. Moreover, the associated pseudo-differential operators will take the form 
    \[
    Au(x)= \frac{1}{(2\pi)^n}\int_{\bbT^n} \int_{\bbR^n} e^{i(x-y)\cdot \zeta} a\left(x+\tau(x-y), \zeta\right) u(y) \,d\zeta dy,
    \]
    where the parameter $\tau\in[0,1]$ allows us to move along different types of quantization. In particular, if we stick to $\tau=0$ and $\kappa=0$, then the classes $\Psi_{\rho, \delta}^{m}\left(\Omega^{0}, \nabla_{LC}, 0\right)$ on $\bbT^n$ coincide with the classes of $1$-periodic pseudo-differential operators with Euclidean symbols first introduced by Agranovich \cite{agra} on the circle $\bbS^1$, and later on generalized to $\bbT^n$. The equivalence of these classes with the local H\"ormander classes $\Psi_{\rho, \delta, loc}^m(\bbT^n)$ was proven by McLean \cite{mcLean}, and later on the second author and Turunnen proved equivalence with their global toroidal classes $\operatorname{Op}(S_{\rho,\delta}^m(\bbT^n\times \bbZ^n))$ \cite{RuT}. 
\end{ex}

\subsection{Function spaces of \texorpdfstring{$\kappa$}{L}-densities}

In this subsection we briefly define the main spaces where we want to study the boundedness of pseudo-differential operators, and state some theorems regarding them that we will use in the next section. The definition of the $L^p$ spaces on densities is standard material (see e.g. \cite{foll}), but for the case of BMO, $H^1$ and Besov spaces it is not. For the latter spaces we have to mention the precise construction of Taylor in \cite{Taylor1}, where he defined local versions of BMO and $H^1$ on complete Riemannian manifolds of bounded geometry; and Triebel construction of Besov spaces also on complete Riemannian manifolds of bounded geometry\cite{Triebel}. These are exactly the spaces we will need for our proofs, but defined on $\kappa$-densities. This is not a difficult task because remembering that $\Omega^\kappa$ is a trivial line bundle, we can use our fixed global section $|dx|$ to transfer the local constructions on $0$-densities (smooth functions $C^\infty(M)$) to the desired spaces. Basically the process is as follows, assume $Y(M)$ is a well-defined space of functions on $M$ then we will define the $\kappa$-density analog $Y(M, \Omega^\kappa)$ as the space of $u\in \Omega^k$ such that $u |dx|^{-\kappa}\in Y(M)$. 

\subsubsection{\texorpdfstring{$L^p$}{L}, BMO and \texorpdfstring{$H^1$}{L} spaces}
In this subsection we assume that $(M,g)$ is a complete Riemannian manifold of bounded geometry. Using densities we can define naturally the intrinsic $L^p$ spaces on manifolds as follows: 
\begin{align*}
    L^p\left(M,\Omega^{1/p}\right) &:= \left\{ u\in \Omega^{1/p}: \left(\int_M |u|^p\right)^{1/p}<\infty \right\},\\
    L^\infty\left(M,\Omega^{0}\right) &:= \left\{ u\in \Omega^{0}: \operatorname{ess\:sup} |u| <\infty \right\}.
\end{align*}
Using our fixed section $|dx|$ of the line bundle $\Omega$ we can define those spaces for any $\kappa$-density:
\begin{align*}
    L^p\left(M,\Omega^{\kappa}\right) &:= \left\{ u \in \Omega^{\kappa}: \left(\int_M \left| u\,|dx|^{\frac{1}{p}-\kappa}\right|^p\right)^{1/p}<\infty \right\},\\
    L^\infty\left(M,\Omega^{\kappa}\right) &:= \left\{ u\in \Omega^{\kappa}: \operatorname{ess\:sup} |u\, |dx|^{-\kappa}| <\infty \right\}.
\end{align*}
\begin{rem}
    Notice that we do not need a Riemannian structure on $M$ to define the $L^p$ spaces, but it would not be the case for the forthcoming spaces. Although once the Riemannian metric is given one can locally write our fixed section as $|dx|=g(x)|dx^1\wedge\cdots\wedge dx^n|$, where from now on $g(x):=|\operatorname{det}g_{ij}(x)|$.  
\end{rem}
Now, we define the spaces $H^1$ and BMO for $\kappa$-densities following \cite{Taylor1}, which is also based in the local construction of Goldberg \cite{Gold} whose spaces were the first step to bring these type of spaces to manifolds. Let $d$ be the associated geodesic distance to the metric tensor $g$. Remember that we denote by $r_0>0$ the injectivity radius of $M$. One of the main consequences of the bounded geometry hypothesis is the existence of uniform partitions of unity, which are key to comparing local norms of function spaces with the global ones defined on manifolds (see e.g. \cite[Lemma 2.16]{bg}). 
\begin{lem}[Uniform partition of unity]\label{lemUnifPar}
    Let $(M,g)$ be a complete Riemannian manifold of bounded geometry. Then for $0<r_2<r_0$ small enough and any $0<r_1\leq r_2<r_0$, $M$ has a countable cover $\{B_{r_1}(x_i)\}_{i\in\bbN}$ and a corresponding smooth and bounded partition of unity $\{\psi_i\}_{i\in \bbN}$ with $\operatorname{supp}\psi_i\subset B_{r_2}(x_j)$ such that
        \begin{enumerate}
            \item $d(x_i,x_j)\geq r_1$ for all $i\neq j$;
            \item There exist a global constant $K$ such that for all $x\in M$ the ball $B_{r_2}(x)$ intersects at most $K$ of the $B_{r_2}(x_i)$.
        \end{enumerate}
\end{lem}
Let us fix such partition of unity $\{\psi_i\}_{i}$, and to fix ideas let us take $r_1=\frac{r_0}{8}$ and $r_2=\frac{r_0}{4}$. For more details about bounded geometry and uniform partitions of unity see e.g. \cite{chee,bg,Taylor1}. 

Let
$$B_\epsilon(x)= \{y\in M : d(x,y)<\epsilon\}$$
be the ball of radius $\epsilon$ centered at $x$ and we denote by 
$$|B_\epsilon(x)| = \int_{B_\epsilon(x)} |dx|$$
its volume. Then we define the average of a $\kappa$-density $u$ on $B_\epsilon(x)$ as 
$$\overline{u_\epsilon}(x)= \frac{1}{|B_\epsilon(x)|}\int_{B_\epsilon(x)} u\,|dx|^{1-\kappa}.$$
Note that, for a fixed $\epsilon$, this is a function $\overline{u_\epsilon}: M \to \mathbb{R}$, i.e. a $0$-density. Hence we define the BMO norm of a $\kappa$-density $u$ by 
\begin{equation}\label{defBMO}
    \|u\|_{\operatorname{BMO}} := \sup_{\substack{\epsilon<\frac{r_0}{4}\\ x\in M}} \frac{1}{|B_\epsilon(x)|} \int_{B_\epsilon(x)} |u(y)\, |dy|^{-\kappa}- \overline{u_\epsilon}(x)|\, |dy| + \sup_{x\in M}\frac{1}{|B_{\frac{r_0}{4}}(x)|} \int_{B_{\frac{r_0}{4}}(x)} |u\,|dx|^{1-\kappa}|,
\end{equation}
therefore we set 
$$\operatorname{BMO}(M, \Omega^\kappa): =\{u\in L^1_{loc}(M,\Omega^\kappa) : \|u\|_{\operatorname{BMO}}<\infty \}.$$
Let $\|\cdot\|_{\operatorname{Lip}}$ denote the norm of Lipschitz functions. On the other hand, let us consider the following set of functions associated to a ball $B_\epsilon(x)$:
\[
\mathcal{F}(B_\epsilon(x))= \left\{\varphi\in C_0^1(B_\epsilon(x)): \|\varphi\|_{\operatorname{Lip}}\leq \frac{1}{\epsilon^{n+1}}\right\}. 
\]
Given $u\in L^1_{loc}(M,\Omega^\kappa)$ let us consider the function
\[
\mathcal{G}_\epsilon u(x) = \sup_{\varphi\in \mathcal{F}(B_\epsilon(x))} \left|\int_{B_\epsilon(x)}\phi(y)u(y)|dy|^{-\kappa}\right|,
\]
so for such $u\in L^1_{loc}(M,\Omega^\kappa)$ we define the following maximal function:
\[
\mathcal{G}^b u(x) = \sup_{\epsilon\,\leq \frac{r_0}{4}} \mathcal{G}_\epsilon u(x). 
\]
Then we define the Hardy space in this setting as
$$H^1(M, \Omega^\kappa): =\{u\in L^1_{loc}(M,\Omega^\kappa) : \mathcal{G}^b u \in L^1(M) \}$$
with norm 
\[
\|u\|_{H^1}=\|\mathcal{G}^b u\|_{L^1}.
\]
\
\begin{rem}
Notice that if in  the previous definitions we take $\kappa=0$ (functions) we recover the exact definitions of \cite{Taylor1}. 
\end{rem}
We state the results concerning these spaces that we are going to need in our proofs, basically the corresponding generalizations of duality and interpolation involving the previous spaces. 
\begin{prop}[Proposition 4.1 in \cite{Taylor1}]\label{dualidadbmo}
    We have that 
    \[
    H^1(M,\Omega^\kappa)' = BMO(M,\Omega^{1-\kappa}).
    \]
\end{prop}
\begin{prop}[Proposition 9.1 in \cite{Taylor1}]\label{interpolacion}
    Given $p\in(1,\infty)$, assume that
    \[
    T:L^p(M,\Omega^\kappa)\to L^p(M,\Omega^\kappa)
    \]
    is bounded. Assume also that 
    \[
    T:L^\infty(M,\Omega^\kappa)\to \operatorname{BMO}(M,\Omega^\kappa)
    \]
    is bounded. Then, for $q\in(p,\infty)$ we have that $T$ is bounded from $L^q(M,\Omega^\kappa)$ to $L^q(M,\Omega^\kappa)$. 
\end{prop}

\subsubsection{Besov spaces}
Again, let $(M,g)$ be a complete Riemannian manifold of bounded geometry. We briefly recall Triebel's \cite{Triebel} construction of Besov spaces adapted to $\kappa$-densities. Let  $\{\psi_i\}_i$ be the uniform partition of unity as before and $\{x_i\}_i$ be the centers of the covering balls (as in Lemma \ref{lemUnifPar}). 
\begin{defn} Let $-\infty<s<\infty$. 
    \begin{enumerate}
        \item Let either $0<p<\infty$, $0<q\leq\infty$ or $p=q=\infty$. Then the Triebel-Lizorkin space is defined as
        \[F_{p,q}^s(M, \Omega^\kappa) = \left\{u\in \Omega^\kappa : \left(\sum_{j=1}^\infty\|\psi_j u|dx|^{-\kappa}(\operatorname{exp}_{x_j})\|_{F_{p,q}^s(\bbR^n)}^p\right)^{1/p}<\infty\right\}.\]
        \item  Let  $0<p\leq\infty$, $0<q\leq\infty$ and $-\infty<s_0<s<s_1<\infty$. Then the Besov space is defined via real interpolation
        \[B_{p,q}^s(M,\Omega^\kappa) = (F_{p,p}^{s_0}(M,\Omega^\kappa), F_{p,p}^{s_1}(M,\Omega^\kappa))_{\Theta, q}\]
        with $s=(1-\Theta)s_0 + \Theta s_1$. 
    \end{enumerate}
\end{defn}
\begin{rem}
    These spaces are of course independent of the choose of partition of unity. Moreover, Triebel proved an equivalent description using heat kernels. 
\end{rem}
For the boundedness of pseudo-differential operators on Besov spaces we will need the following important results by Triebel, which resemble classical properties on $\bbR^n$:

\begin{thm}[Theorem 4 in \cite{Triebel}]
    Let $1<p<\infty$ and $-\infty <s < \infty$. Then
    \[
    L_s^p(M,\Omega^\kappa)=F_{p,2}^s(M,\Omega^\kappa). 
    \]
\end{thm}

\begin{thm} [Theorem 5 in \cite{Triebel}]
    Let $0<p<\infty, 0<q_0\leq \infty, 0<q_1\leq \infty$ and $-\infty< s_0 <s_1 < \infty$. Let $0<q\leq \infty$,  $0<\Theta<1$, and 
    \[s = (1-\Theta)s_0+\Theta s_1. \]
    Then 
    \[
    (B_{p,q_0}^{s_0}(M,\Omega^\kappa), B_{p,q_1}^{s_1}(M,\Omega^\kappa))_{\Theta, q} = (F_{p,q_0}^{s_0}(M,\Omega^\kappa), F_{p,q_1}^{s_1}(M,\Omega^\kappa))_{\Theta, q} = B_{p,q}^s(M,\Omega^\kappa).
    \]
\end{thm}
As an immediate corollary of these two theorems we get the following result. 
\begin{cor}
\label{interoplacionSobolevBesov}
     Let $1<p<\infty$ and $-\infty< s_0 <s_1 < \infty$. Let $0<q\leq \infty$, $0<\Theta<1$, and 
    \[s = (1-\Theta)s_0+\Theta s_1. \]
    Then
    \[
    (L_{s_0}^p(M,\Omega^\kappa), L_{s_1}^p(M,\Omega^\kappa))_{\Theta, q} = B_{p,q}^s(M,\Omega^\kappa).
    \]
\end{cor}

\subsection{Bessel potentials}\label{BesselPotentials} In this subsection we illustrate how we can use the classical Bessel potentials for the estimation of norms of Safarov pseudo-differential operators. 
Let  $(M,g)$ be a complete Riemannian manifold, recall that the Laplace-Beltrami operator associated to the metric $g$ is given by 
\[
Lu(x) = \frac{1}{g(x)}\partial_{x^i}\left(g(x) g^{ij}(x)\partial_{x^j}u(x)\right),
\]
where $g_{ij}$ and $g^{ij}$ are the metric and inverse metric coefficients, respectively, and $g(x):=|\operatorname{det}g_{ij}(x)|$. Notice that this operator can be naturally defined on $\kappa$-densities (actually in any vector bundle). We define the Bessel potential as follows:
\[
\mathfrak{B} = (1-L)^{1/2},
\]
and we are interested in determining the classes of pseudo-differential operators to which its powers belong to, i.e. we want to know when $\mathfrak{B}^\lambda\in \Psi_{\rho, \delta}^{m}\left(\Omega^{\kappa}, \nabla\right)$ for $\lambda\in \bbR$. It is well known that for the classical classes $\mathfrak{B}^\lambda\in \Psi_{1, 0, loc}^{\lambda}\left(\Omega^{\kappa}\right)$ for any $\lambda\in \bbR$, thus $\mathfrak{B}^\lambda\in \Psi_{1, 0}^{\lambda}\left(\Omega^{\kappa},\nabla\right)$ for any connection $\nabla$  and any $\lambda\in \bbR$ since Safarov classes coincide with the classical ones when $\rho>1/2$. Therefore, the obvious inclusion $\Psi_{1, 0}^{m}\left(\Omega^{\kappa}, \nabla\right)\subseteq \Psi_{\rho, \delta}^{m}\left(\Omega^{\kappa}, \nabla\right)$ for $0\leq \delta<\rho\leq 1$ give us that 
\begin{equation}\label{besselEnSafarov}
    \mathfrak{B}^\lambda\in \Psi_{\rho, \delta}^{\lambda}\left(\Omega^{\kappa}, \nabla\right) \text{ for all } \nabla \text{ and }\lambda\in \bbR. 
\end{equation}
Now, combining the latter result with Theorem \ref{producto}  will allow us to transfer norm estimates of pseudo-differential operators to the ones of the Bessel potentials. 
\begin{rem}
\label{NormaBessel}
    Let $(X, \|\cdot\|_X), (Y, \|\cdot\|_Y)$ be Banach spaces and suppose $A\in \Psi_{\rho, \delta}^{m}\left(\Omega^{\kappa}, \nabla\right)$ is defined on them, i.e. $A: X\to Y$ is well-defined. Hence for $x\in X$ we have that
    \[
    \|Ax\|_Y = \|(A\mathfrak{B}^{-m})\circ \mathfrak{B}^{m}x\|_Y \leq \|A\mathfrak{B}^{-m}\|_{op} \|\mathfrak{B}^{m}x\|_Y,
    \]
where $A\mathfrak{B}^{-m}$ is a pseudo-differential operator of order zero by Theorem \ref{L2bound} and $\|\cdot\|_{op}$ is the operator norm. In this way we could transform the problem to know the estimation of zero order operators and Bessel potentials.
\end{rem}
To close this section, we quickly rewrite the Sobolev embedding theorem in terms of the boundedness of Bessel potentials. On complete Riemannian manifolds of bounded geometry the Sobolev embedding theorem  states (see e.g. \cite{her}) that for any real numbers $1\leq p<q$, $0\leq m<k$, satisfying 
\begin{equation}\label{sobEquaOrd}
    k-m = n\left(\frac{1}{p}-\frac{1}{q}\right),
\end{equation} one has the continuous inclusion $L_k^p\subset L_m^q$ of Sobolev spaces. Of course the latter result can be rewritten in terms of Bessel potentials, and moreover using \cite[Theorem 4.2]{stri} and Remark \ref{NormaBessel} one can replace the equality in condition \eqref{sobEquaOrd} by an inequality. The corresponding statement is as follows:
\begin{thm}
\label{sobolev}
    Let $1\leq p<q$ and $s\in\bbR$. Then $\mathfrak{B}^{-s}$ is bounded from $L^p(M)$ to $L^q(M)$ if 
    \[
    n\left(\frac{1}{p}-\frac{1}{q}\right) \leq s.
    \]
\end{thm}

\section{Main result and consequences}
\label{ultima}
In this section we prove the main results presented in the introduction, namely the $L^p-L^p$ boundedness of pseudo-differential operators belonging to Safarov global classes. Which together with the stated theorems of Section \ref{prelimi} will imply boundedness on Sobolev and Besov spaces, and some particular $L^p-L^q$ boundedness. In order not to overload the notation, sometimes, we are going to omit the dependence of $\Omega^\kappa$ on spaces, but remember that we are always dealing with $\kappa$-densities. Henceforth, given a symbol $a\in S_{\rho, \delta}^m(\nabla)$ we denote by $a(X,D)$ the corresponding pseudo-differential operator in the class $\Psi_{\rho, \delta}^{m}\left(\Omega^{\kappa}, \nabla\right)$. Henceforth, we assume that $(M,g,\nabla)$ is a complete Riemannian manifold of bounded geometry (see Definition \ref{boundedGeometry}), so remember that $r_0>0$ refers to the injectivity radius of $M$ and $g(x)$ is denoting the function $\sqrt{|\operatorname{det}g_{ij}(x)|}$. 
\subsection{\texorpdfstring{$L^p$}{L}-boundedness}
We follow the strategy of the $\bbR^n$ case proof done by Fefferman, namely we prove the $L^p$-boundedness for a particular order $\theta$ via interpolation between $L^2-L^2$ and $L^\infty-\operatorname{BMO}$, and then extend the the range of $\theta$ via Stein's complex interpolation. Key tools of the proof are Lemma \ref{lem2} and Theorem \ref{bmo}, which allow us to control the problem on annular neighborhoods appearing in partitions of unity. We often use compactness in the fact that it implies bounded geometry, so one can handle Christoffel symbols and metric coefficients. 

We start with a simple estimate, a consequence of the H\"older inequality. 
\label{S3}
\begin{lem}
\label{lem1}
Let $A\in \Psi_{\rho, \delta}^{m}\left(\Omega^{\kappa}, \nabla\right)$ such that its Schwartz kernel satisfies
\[
    \underset{x\in M}{\operatorname{ess\:sup}}\|\mathscr{A}(x, \cdot)\|_{L^1(M)} <\infty.
\]
Then $A$ extends to a bounded operator from $L^\infty(M,\Omega^\kappa)$ to $L^\infty(M,\Omega^\kappa)$, and 
\[
\|A u\|_{L^\infty(M)}\leq \underset{x\in M}{\operatorname{ess\:sup}}\|\mathscr{A}(x, \cdot)\|_{L^1(M)}\| u\|_{L^\infty(M)}, \quad \forall u\in L^\infty(M).
\]
\end{lem}
\begin{proof}
Using the H\"older inequality we get
$$|A u(x)| = \left|\int_M \skernel(x,y) u(y) |dy| \right|\leq \int_M\left| \skernel(x,y)u(y) \right| |dy|\leq \|\skernel(x,\cdot)\|_{L^1(M)}\|u\|_{L^\infty(M)}.$$ 
Taking the supremum, the result follows.
\end{proof}
Let us briefly remind the content of Remark \ref{kernels}. For $a\in S_{\rho,\delta}^m(\nabla)$, the associated pseudo-differential operator $a(X,D)$ can be written as the local part plus the global part
\[
a(X,D) = a(X,D)_{loc} + a(X,D)_{glo},
\]
so to have global estimates we will require Assumption \ref{assu} on the global part. Indeed, let us handle the global part as follows: 
\begin{lem}\label{lemGlob}
    Let $a\in S_{\rho,\delta}^m(\nabla)$ with $0\leq\delta<\rho\leq 1$, and let $1\leq p\leq \infty$. Suppose that the kernel of $a(X,D)_{glo}$ satisfies Assumption \ref{assu}, then $a(X,D)_{glo}$ is bounded from $L^p(M,\Omega^\kappa)$ to $L^p(M,\Omega^\kappa)$.
\end{lem}
\begin{proof}
As in the case of the Lemma \ref{lem1}, using the H\"older inequality, the $L^p$ boundedness of $a(X,D)_{glo}$ will hold if the following quantities of the kernel are bounded
\[
\underset{x\in M}{\operatorname{ess\:sup}}\|K(x, \cdot)\|_{L^1(M)}<\infty,\quad \underset{y\in M}{\operatorname{ess\:sup}}\|K(\cdot,y)\|_{L^1(M)}<\infty, 
\]
where $K$ is defined in \eqref{kernelDef}. Let us treat the first term, the second one is similar. Since $K$ is the kernel away from the diagonal, we want to estimate the kernel outside of a geodesic ball $B_\epsilon(x_0)$ for $x_0$ fixed and $\epsilon<r_0$, where $r_0$ is the injectivity radius. Suppose $K_0>0$, thus using Lemma \ref{volGrowth}, Assumption \ref{assu} and the standard polar decompostion on complete Riemannian manifolds (see e.g. \cite[Section III]{polar}) we get 
\begin{align*}
    \|K(x, \cdot)\|_{L^1(M)} & = \int_{M\backslash B_\epsilon(x_0)} |K(x,y)| |dy|^\kappa \\
    & \leq C\int_{M\backslash B_\epsilon(x_0)} e^{-W d(x,y)} |dy|^\kappa\\
    &\leq C \int_{\epsilon}^\infty (1+r)^{\mu_0}e^{(K_0-W)r}  \, dr<\infty \text{ (because $W>K_0)$. }
\end{align*}
The calculation is similar if $K_0=0$ because we have the condition $L>\mu_0+n$ to control the polynomial growth of the ball. The result follows.
\end{proof}
Henceforth, when $a(X,D)_{loc}$ appears in the boundedness results that would mean local boundedness, i.e. $a(X,D)_{loc}$ bounded from $L^p$ to $L^p$ means $a(X,D)$ bounded from $L^p_{comp}$ to $L^p_{loc}$. Moreover, by making Assupmtion \ref{assu} we will have in addition global boundedness due to Lemma \ref{lemGlob}. 

The following lemma is a fundamental piece to prove the $L^p$-estimates since it allows us to exploit the use of balls/annular partitions of unity on the dual variable space. 
\begin{lem}
\label{lem2}
Let $a\in S_{\rho, \delta}^{-\frac{n(1-\rho)}{2}}(\nabla)$. Suppose that for any point $x$ in an open set, the support of $a(x,\cdot)$ is the ring/ball 
$$\mathfrak{R}_x=\{R\leq |\eta|_x \leq 3R\} \hspace{0.5cm}\text{ or } \hspace{0.5cm}\mathfrak{B}_x=\{|\eta|_x \leq R\},$$
for some fixed $R>0$ satisfying the following condition:
\begin{equation}
\label{condR}
    R\geq \frac{1}{r_0^{1/\rho}}.
\end{equation}
Then $a(X, D)_{loc}$ is locally bounded from $L^\infty(M,\Omega^\kappa)$ to $L^\infty(M,\Omega^\kappa)$, and for $l\geq n/2$ we have 
$$\|a(X, D)_{loc}u\|_{L^\infty(M)}\leq C \|a\|_{l; S_{\rho, \delta}^{-\frac{n(1-\rho)}{2}}}\|u\|_{L^\infty(M)}.$$
Moreover, if $a\in S_{\rho,\delta}^{-\frac{n(1-\rho)}{2}}(\nabla)_{glo}$ then $a(X,D)$ is globally bounded from $L^\infty(M,\Omega^\kappa)$ to $L^\infty(M,\Omega^\kappa)$. 
\end{lem}
\begin{proof}
The result will follow from Lemma \ref{lem1}, so we focus on proving the condition 
\[
\underset{x\in M}{\operatorname{ess\:sup}}\|\mathscr{A}(x, \cdot)\|_{L^1(M)} <\infty.
\]
We will just write $a$ for $a_{loc}$. Let $x\in M$ and set $\epsilon = R^{-\rho}$. The condition \eqref{condR} implies that $B_{\epsilon}(x)$ is a geodesic ball. Thus we split $M$ as follows: 
$$M= B_\epsilon(x) \cup (B_\epsilon(x))^c,$$
and since outside the small neighbourhood $B_{\epsilon}(x)\times B_{\epsilon}(x)$ the operator $a(X,D)$ is represented by a smooth kernel, we only care on estimating the $L^1$ norm on the geodesic ball. Thus for $y\in B_{r_0}(x)$, using Remark \ref{remCoNo}, we are interested on bounding 
\begin{align}
\label{cosaAcotar}
    \begin{split}
         \int_{B_{r_0}(x)} \left| \mathscr{A}(x, y) |dy|^\kappa\right| & =\frac{1}{(2\pi)^n}\int_{B_{r_0}(x)} \left|\int_{T_x^*M}  \Upsilon_x^{1-\kappa}(y)e^{i\varphi_0(x,\zeta,y)} a\left(x, \zeta\right)g^{-1}(x) d\zeta\right| g(y)^\kappa dy\\
         &= \frac{1}{(2\pi)^n}\int_{B_{r_0}(x)} \left|\int_{\bbR^n}  \Upsilon_0^{1-\kappa}(y)e^{iy\cdot \zeta} a\left(0, \zeta\right) d\zeta\right| g(y)^\kappa dy,
    \end{split}
\end{align}
because in normal coordinates the point $x$ corresponds to the origin in $\bbR^n$, moreover the inner integral is over $\mathfrak{R}_0$ or $\mathfrak{B}_0$ since that is the support of the symbol $a$. We are in position to apply a similar procedure as in \cite{Fefferman}, but we have to be careful with the weight $\Upsilon_0^{1-\kappa}(y)$. This is a smooth function that depends on the Christoffel symbols $\Gamma_{k j}^{i}$ and since $\nabla$ is of bounded geometry we can bound it uniformly. Therefore, using the Cauchy-Schwarz inequality, the fact that we have bounded geometry and the bounds for the symbol, we get for \eqref{cosaAcotar} in $B_\epsilon(0)$, that
\begin{align*}
    \int_{B_\epsilon(x)}\left| \mathscr{A}(x, y) |dy|^\kappa\right| & \leq \frac{1}{(2\pi)^n}\int_{B_\epsilon(0)}\left| \Upsilon_0^{1-\kappa}(y)\right| \left|\int_{\bbR^n}  e^{iy\cdot \zeta} a\left(0, \zeta\right) d\zeta\right|g(y)^\kappa dy\\
    &\leq \frac{1}{(2\pi)^n}\left(\int_{B_\epsilon(0)}\left| \Upsilon_0^{1-\kappa}(y)\right|^2 g(y)^{2\kappa} dy\right)^{1/2}\\
    &\times\left(\int_{B_\epsilon(0)} \left|\int_{T_x^*M}  e^{iy\cdot \zeta} a\left(0, \zeta\right) d\zeta \right|^2 dy\right)^{1/2}\\
    &\leq C_n \left(\int_{B_\epsilon(0)} dy\right)^{1/2} \left(\int_{\bbR^n} \left|\int_{\bbR^n}  e^{iy\cdot \zeta} a\left(0, \zeta\right) d\zeta \right|^2 dy\right)^{1/2}\\
    &\leq C_n \epsilon^{n/2} \left(\int_{\bbR^n}  \left| a\left(0, \zeta\right)\right|^2 d\zeta \right)^{1/2}\\
    & \leq  C_n R^{-n\rho/2}  \|a\|_{0; S_{\rho, \delta}^{-\frac{n(1-\rho)}{2}}}   \left(\int_{\mathfrak{R}_0} \langle\zeta\rangle_{0}^{-n(1-\rho)} d\zeta \right)^{1/2} \\
    &=C_n R^{-n\rho/2}\|a\|_{0; S_{\rho, \delta}^{-\frac{n(1-\rho)}{2}}} \left(\int_{\mathfrak{R}_0} \frac{1}{(1+|\zeta|_0^2)^{n(1-\rho)/2}}d\zeta \right)^{1/2}\\
    &\leq C_n R^{-n\rho/2}\|a\|_{0; S_{\rho, \delta}^{-\frac{n(1-\rho)}{2}}} R^{n\rho/2}\text{ (since }g^{ab}(0)=\delta_{ab})\\
    &\leq C_n \|a\|_{0; S_{\rho, \delta}^{-\frac{n(1-\rho)}{2}}}<\infty,
\end{align*}
where $C_n$ is independent of $R$ and $a$. In the previous calculation we deal with the ring $\mathfrak{R}_x$, but the case of the ball $\mathfrak{B}_x$ is analogous since the integral 
\[
\int_{\mathfrak{B}_0} \frac{1}{(1+|\zeta|_0^2)^{n(1-\rho)/2}}d\zeta \leq C R^{n\rho}
\]
as well. We notice that the upper bound is independent of $x$, thus by taking supremum we prove the desired condition and the result follows from Lemma \ref{lem1}. Finally, the global boundedness follows from this and Lemma \ref{lemGlob}.
\end{proof}
Now we arrive to the key part of the proof, the $L^\infty-\operatorname{BMO}$ boundedness. 
\begin{thm}
\label{bmo}
Let at least one of the conditions (1)-(3) of Theorem \ref{producto} be fulfilled. Let $a\in S_{\rho, \delta}^{-\frac{n(1-\rho)}{2}}(\nabla)$. Then $a(X, D)_{loc}$ is locally bounded from $L^\infty(M,\Omega^\kappa)$ to $\operatorname{BMO}(M,\Omega^\kappa)$, and for $l> n/2$ we have 
$$\|a(X, D)_{loc}u\|_{\operatorname{BMO}(M)}\leq C \|a\|_{l; S_{\rho, \delta}^{-\frac{n(1-\rho)}{2}}}\|u\|_{L^\infty(M)}.$$
\end{thm}
\begin{proof}
To simplify the notation, we will write $a$ for $a_{loc}$. Let $u\in L_{comp}^\infty(M)$ and let us fix $B_r(x_0)$ with $r<\frac{r_0}{4}$, where $r_0$ is the injectivity radius of $M$ (remember that $\frac{r_0}{4}$ is the radius appearing in the definition of the BMO space). We are going to split the covariable support of the symbol in order to apply Lemma \ref{lem2} on each piece. Let $A>0$ be a constant such that 
\[
A\geq \frac{1}{4r_0^{1/\rho}}. 
\]
Now, let $\psi$ be a function such that $a^0(y,\zeta) := a(y,\eta)\psi(y,\eta)$ has support in $S_0=\{(y,\eta)\in T^*M: |\eta|_{y} \leq 2Ar_0r^{-1} \}$, so that $a^1:= a-a^0$ is supported in $S_1=\{(y,\eta)\in T^*M: |\eta|_{y} \geq Ar_0r^{-1} \}$. For instance, take an adequate bump function $\tilde{\psi}$ on $\bbR$ and define $\psi(y,\eta)=\tilde{\psi}(r|\eta|_y)$, where $r$ is the fixed radius of our ball. Therefore, $a = a^0 + a^1$ and we have the following estimates for the splitting
\begin{equation}
\label{normasDeLaParticion}
    \|a^0\|_{l; S_{\rho, \delta}^{-\frac{n(1-\rho)}{2}}}, \|a^1\|_{l; S_{\rho, \delta}^{\frac{-n(1-\rho)}{2}}} \leq C_l \|a\|_{l; S_{\rho, \delta}^{-\frac{n(1-\rho)}{2}}}\text{ for every }l\geq 1. 
\end{equation}
Our aim is to estimate $\|a(X, D)u\|_{\operatorname{BMO}}$ for the fixed ball at the beginning, i.e. we want to bound for $r$
\begin{equation}
\label{BMOCosaAcotar}
 \frac{1}{|B_r(x_0)|} \int_{B_r(x_0)} \left|a(X, D)u(y)\, |dy|^{-\kappa}- \overline{(a(X, D)u)}_r(x_0)\right|\, |dy|,   
\end{equation} and for the biggest possible ball we want to control just the average
\begin{equation}
\label{BMOCosaAcotar2}
    \frac{1}{|B_{\frac{r_0}{4}}(x_0)|} \int_{B_{\frac{r_0}{4}}(x_0)} \left|a(X, D)u(y)\, |dy|^{1-\kappa}\right|.
\end{equation}
So using the fact that $a = a_0 + a_1$ and triangle inequality we get that \eqref{BMOCosaAcotar} has the following upper bound
\begin{align}
\label{1}
\begin{split}
    \frac{1}{|B_r(x_0)|} \int_{B_r(x_0)} |a^0(X, D)u(y) |dy|^{-\kappa}- &\overline{(a(X, D)u)}_r(x_0)| |dy| \\
    &+ \frac{1}{|B_r(x_0)|} \int_{B_r(x_0)} |a^1(X, D)u(y) |dy|^{-\kappa}| |dy|,
\end{split}
\end{align}
and on the other hand \eqref{BMOCosaAcotar2} has the upper bound
\begin{align}
\label{22}
    \begin{split}
         &\frac{1}{|B_{\frac{r_0}{4}}(x_0)|} \int_{B_{\frac{r_0}{4}}(x_0)} \left|a^0(X, D)u(y)\, |dy|^{1-\kappa}\right| + \frac{1}{|B_{\frac{r_0}{4}}(x_0)|} \int_{B_{\frac{r_0}{4}}(x_0)} \left|a^1(X, D)u(y)\, |dy|^{1-\kappa}\right|\\
         & \leq \|a^0(X, D)u(y)\|_{L^\infty(M)} + \frac{1}{|B_{\frac{r_0}{4}}(x_0)|} \int_{B_{\frac{r_0}{4}}(x_0)} \left|a^1(X, D)u(y)\, |dy|^{1-\kappa}\right|,
    \end{split}
\end{align}
hence we focus on estimating these upper bounds. Notice that the terms involving $a^1$ are of the same type i.e. both are averages on geodesic balls but we include the biggest possible ball of radius $\frac{r_0}{4}$, thus later on when we deal with this term the proof will cover the case of $0<r\leq \frac{r_0}{4}$ instead of just $0<r< \frac{r_0}{4}$.

First, let us treat the term corresponding to $a^0$. We recognize that we can bound immediately the first term of \eqref{22} because by construction $a^0$ satisfies the hypothesis of Lemma \ref{lem2}, thus 
\begin{equation}
\label{23}
    \|a^0(X, D)u\|_{L^\infty(M)}\leq C \|a\|_{l; S_{\rho, \delta}^{-\frac{n(1-\rho)}{2}}}\|u\|_{L^\infty(M)}.
\end{equation}
The tricky term is the first one of \eqref{1}. We want to use the mean value theorem, so we would like to study the operator $\partial_{x^k} a^0(X, D)$, which is defined by composition. Then, finish the estimation again using Lemma \ref{lem2}. Calculating $\partial_{x^k} a^0(X, D)$ in normal coordinates gives us 
\begin{align*}
    \partial_{x^k} a^0(X, D)u(x) &= \partial_{x^k}\left(\frac{1}{(2\pi)^n}\int_{\bbR^n} \int_{T_{x_0}^*M}  \Upsilon_x^{1-\kappa}(y)e^{i(x-y)\cdot \zeta} a^0\left(x, \zeta\right) u(y) d\zeta g(y)dy\right)\\
    &=\frac{1}{(2\pi)^n}\int_{\bbR^n} \int_{T_{x_0}^*M}  \Upsilon_x^{1-\kappa}(y)e^{i(x-y)\cdot \zeta} \left[(\kappa-1) \Upsilon_x(y)(\partial_{x^k} \Upsilon_y(x))a^0\left(x, \zeta\right)\right.\\
    &\hspace{3cm}\left.+i\zeta_ka^0\left(x, \zeta\right) +\partial_{x^k}a^0\left(x, \zeta\right)\right] u(y) d\zeta g(y)dy,
\end{align*}
where we used the property $\Upsilon_x(y) = \left(\Upsilon_y(x)\right)^{-1}$ to compute $\partial_{x^k}\Upsilon_x^{1-\kappa}(y)$. Thus, using formula \eqref{determiTrans} we get that $\partial_{x^k} a^0(X, D)$ is the pseudo-differential operator given by the amplitude (in normal coordinates)
\begin{align*}
    a'(y;x,\zeta)&= (\kappa-1) \Upsilon_x(y)\Gamma_{kj}^j(x)a^0\left(x, \zeta\right)+i\zeta_ka^0\left(x, \zeta\right) +\partial_{x^k}a^0\left(x, \zeta\right)\\
    &= (\kappa-1) \Upsilon_x(y)\Gamma_{kj}^j(x)a^0\left(x, \zeta\right)+i\zeta_ka^0\left(x, \zeta\right) +\nabla_k a^0\left(x, \zeta\right)\\
    &-\Gamma_{kb}^a(x)\zeta_a\partial_{\zeta_b}a^0\left(x, \zeta\right),
\end{align*}
where the terms having Christoffel symbols can be controlled because $\nabla$ is of bounded geometry. Moreover, notice that the $\zeta$ support of $a'$ is the same as the one of $a^0$. Here we identify two cases depending on the value of $\kappa$, equal to $1$ and different from $1$. First, let us suppose $\kappa\neq 1$. By Proposition \ref{AmpliToSym} we have that up to order 0 (which is enough since we are already looking for the symbol of a derivative) there exist a symbol $t^0(x,\zeta)$ such that 
\begin{align*}
    t^0(x,\zeta) &=  a'(y;x,\zeta) \bigg|_{y=x}\\
    &= \bigg((\kappa-1) \Upsilon_x(y)\Gamma_{kj}^j(x)a^0\left(x, \zeta\right)+i\zeta_ka^0\left(x, \zeta\right) +\nabla_k a^0\left(x, \zeta\right)\\
    &-\Gamma_{kb}^a(x)\zeta_a\partial_{\zeta_b}a^0\left(x, \zeta\right)\bigg) \bigg|_{y=x}\\
    &=(\kappa-1)\Gamma_{kj}^j(x)a^0\left(x, \zeta\right)+i\zeta_ka^0\left(x, \zeta\right) +\nabla_k a^0\left(x, \zeta\right)-\Gamma_{kb}^a(x)\zeta_a\partial_{\zeta_b}a^0\left(x, \zeta\right),
    \end{align*}
where the presence of Christoffel symbols is not a problem since $\nabla$ is of bounded geometry. Therefore, we have the estimate
\begin{align}
\label{controlTaylor1}\begin{split}
     |t^0(x,\zeta)| &= \left|(\kappa-1)\Gamma_{kj}^j(x)a^0\left(x, \zeta\right)+i\zeta_ka^0\left(x, \zeta\right) +\nabla_k a^0\left(x, \zeta\right)-\Gamma_{kb}^a(x)\zeta_a\partial_{\zeta_b}a^0\left(x, \zeta\right) \right|\\
    &\leq C_\kappa|\Gamma_{kj}^j(x)a^0\left(x, \zeta\right)|+|\zeta_ka^0\left(x, \zeta\right)| +|\nabla_k a^0\left(x, \zeta\right)|+|\Gamma_{kb}^a(x)\zeta_a\partial_{\zeta_b}a^0\left(x, \zeta\right)|\\
    &\leq C_\kappa\langle\zeta\rangle_x^{\frac{-n(1-\rho)}{2}}+\langle\zeta\rangle_x^{\frac{-n(1-\rho)}{2} +1} + \langle\zeta\rangle_x^{\frac{-n(1-\rho)}{2} +\delta}+C\langle\zeta\rangle_x^{\frac{-n(1-\rho)}{2} +1-\rho}\\
    &\leq C_\kappa \langle\zeta\rangle_x^{\frac{-n(1-\rho)}{2} + 1}\\
    &\leq \frac{C_\kappa}{r} \langle\zeta\rangle_x^{\frac{-n(1-\rho)}{2}},
\end{split} 
\end{align}
where the last inequality holds since in this case $|\zeta|_x\leq A\frac{r_0}{r}$. Thus, we get that for all $l$
\begin{align*}
    \|t^0\|_{l; S_{\rho, \delta}^{-\frac{n(1-\rho)}{2}}}\leq \frac{C_{l,\kappa}}{r} \|a^0\|_{l; S_{\rho, \delta}^{-\frac{n(1-\rho)}{2}}}.
\end{align*}
Hence utilizing the previous inequality and Lemma \ref{lem2} we get 
\begin{align*}
    \|\partial_{x^k} a^0(X, D)u\|_{L^\infty(M)}&  = \|t^0(X,D)u\|_{L^\infty(M)}\\
    &\leq C \|t^0\|_{l; S_{\rho, \delta}^{-\frac{n(1-\rho)}{2}}}\|u\|_{L^\infty(M)}\\
    &\leq \frac{C_{l,\kappa}}{r}\|a_0\|_{l; S_{\rho, \delta}^{-\frac{n(1-\rho)}{2}}}\|u\|_{L^\infty(M)}. 
\end{align*}
On the other hand, if $\kappa=1$ we have that $\partial_{x^k} a^0(X, D)$ is the pseudo-differential operator given by the symbol (in normal coordinates)
$$a'(x,\zeta)= i\zeta_ka^0\left(x, \zeta\right) +\partial_{x^k}a^0\left(x, \zeta\right),$$
and we can estimate its absolute value as follows: 
\begin{align}
    \begin{split}
        |a'(x,\zeta)| & = |i\zeta_ka^0\left(x, \zeta\right) +\partial_{x^k}a^0\left(x, \zeta\right)|\\
        &\leq |\zeta_ka^0\left(x, \zeta\right)| + |\Gamma_{kj}^i\zeta_i\partial_{\zeta_j}a^0\left(x, \zeta\right)| + |\nabla_ka^0\left(x, \zeta\right)|\\
        &\leq \langle\zeta\rangle_x^{\frac{-n(1-\rho)}{2} + 1}+C\langle\zeta\rangle_x^{\frac{-n(1-\rho)}{2} +1-\rho} + \langle\zeta\rangle_x^{\frac{-n(1-\rho)}{2} +\delta}\\
        &\leq C \langle\zeta\rangle_x^{\frac{-n(1-\rho)}{2} + 1}. 
    \end{split}
\end{align}
As in the other case we obtain for all $l$
\begin{align*}
    \|a'\|_{l; S_{\rho, \delta}^{-\frac{n(1-\rho)}{2}}}\leq \frac{C_l}{ r} \|a_0\|_{l; S_{\rho, \delta}^{-\frac{n(1-\rho)}{2}}},
\end{align*}
and finally obtaining  
\begin{equation}\label{a0estim}
    \|\partial_{x^k} a^0(X, D)u\|_{L^\infty(M)} \leq \frac{C_l}{r}\|a_0\|_{l; S_{\rho, \delta}^{-\frac{n(1-\rho)}{2}}}\|u\|_{L^\infty(M)}.
\end{equation}
Notice that we achieve the same bound in the two cases, so we proceed as just one case. From the mean value theorem for Riemannian manifolds and the previous estimate \eqref{a0estim} ($x$ was a silent variable) we get that 
\begin{align*}
    |a^0(X, D)u(y) |dy|^{-\kappa}- \overline{a(X, D)u_\epsilon}(x_0)|&\leq d(x_0,y) |\operatorname{grad}_y \{a^0(X, D)u(y)\} |\\
    &\leq r \sqrt{g_{ab}g^{ia}g^{ib}(\partial_{y^i}a^0(X, D)u(y))^2)}\\
    &\leq r |\partial_{y^i}a^0(X, D)u(y)|\\
    & \leq C \|a_0\|_{l; S_{\rho, \delta}^{-\frac{n(1-\rho)}{2}}}\|u\|_{L^\infty(M)}.
\end{align*}
Integrating the last inequality and taking into account the property \eqref{normasDeLaParticion} we have
\begin{align}
    \label{2}
    \begin{split}
        \frac{1}{|B_r(x_0)|} \int_{B_r(x_0)} |a^0(X, D)u(y) |dy|^{-\kappa}- \overline{(a(X, D)u)}_r(x_0)| |dy|&\leq C \|a^0\|_{l; S_{\rho, \delta}^{-\frac{n(1-\rho)}{2}}}\|u\|_{L^\infty(M)}\\
        &\leq C \|a\|_{l; S_{\rho, \delta}^{-\frac{n(1-\rho)}{2}}}\|u\|_{L^\infty(M)},
    \end{split}  
\end{align}
which finishes the procedure for $a^0$. 

Let us move on to the case of $a^1$. We are going to exploit the fact that we are performing everything local since $u$ has compact support and $a^1(X,D)u(x)$ is localized as well, both with small enough support due to partitions of unity, if needed. We choose a function $\phi\in C^\infty(\bbR^n)$ such that: $0\leq \phi\leq 10$, $\phi\geq 1$ in $B_r(0)$ and $\operatorname{supp} \widehat{\phi}(\nu)\subset \{|\nu|\leq r^{1/\rho}\}$, where $\widehat{\phi}(\nu)$ means Fourier transform. Since $\phi\geq 1$ in $B_r(0)$ we have the following inequality in local coordinates
\begin{align}\label{pasarAProducPhi}
    \begin{split}
        \frac{1}{|B_r(0)|} \int_{B_r(0)} |a^1(X, D)u(y)& |dy|^{-\kappa}| |dy|\\
        &\leq \frac{1}{|B_r(0)|} \int_{B_r(0)} |\phi(x) a^1(X, D)u(y) |dy|^{-\kappa}| |dy|,
    \end{split}
\end{align}
so we focus on estimating the right hand side. We are going to use its commutator to estimate the $\kappa$-density $\phi(x)\cdot a^1(X,D)u(x)$, namely 
\[
\phi(x)\cdot a^1(X,D)u(x) = \underbrace{a^1(X,D)(\phi u)(x)}_\text{I} + \underbrace{\left[\phi, a^1(X,D)\right]u(x)}_\text{II}.
\]
We estimate each term separately. For I we are going to use the Bessel potentials $\mathfrak{B}^\lambda$ defined in Subsection \ref{BesselPotentials} to deal with the $\phi$ dependence. Using the Cauchy-Schwarz inequality in \eqref{pasarAProducPhi} one can transfer the boundedness we are looking for by an $L^2$ one, namely
\begin{align}
\label{estimacionparaI}
\begin{split}
    \frac{1}{|B_r(0)|} \int_{B_r(0)} &|a^1(X, D)(\phi u)(y) |dy|^{-\kappa}| |dy|\leq \\
    &\leq \frac{1}{|B_r(0)|} \left(\int_{B_r(0)} |dy|\right)^{1/2}\left(\int_{B_r(0)} |a^1(X, D)(\phi u)(y) |dy|^{-\kappa}|^2|dy|\right)^{1/2}\\
    & \leq \frac{1}{|B_r(0)|^{1/2}} \|a^1(X, D)(\phi u)\|_{L^2(\bbR^n)},
\end{split}
\end{align}
thus we just need to estimate the $L^2(\bbR^n)$ norm of $a^1(X, D)(\phi u)$. We follow the idea of Remark \ref{NormaBessel}. Consider the operator  $\mathfrak{B}^{\frac{n(1-\rho)}{2}}$, then by \eqref{besselEnSafarov} we have that $\mathfrak{B}^{\frac{n(1-\rho)}{2}}\in \Psi_{\rho, \delta}^{n(1-\rho)/2}\left(\Omega^{\kappa}, \nabla\right)$ and also that its inverse $\mathfrak{B}^{-\frac{n(1-\rho)}{2}}$ is an operator belonging to the class $\Psi_{\rho, \delta}^{-n(1-\rho)/2}\left(\Omega^{\kappa}, \nabla\right)$ (see \cite[Theorem 10.4]{Safarov}). Therefore we can rewrite I in the following way
$$\operatorname{I} = (a^1(X,D)\circ \mathfrak{B}^{\frac{n(1-\rho)}{2}})(\mathfrak{B}^{-\frac{n(1-\rho)}{2}}(\phi u)),$$
where $a^1(X,D)\circ \mathfrak{B}^{\frac{n(1-\rho)}{2}}\in \Psi_{\rho, \delta}^{0}\left(\Omega^{\kappa}, \nabla\right)$. Theorem \ref{L2bound} implies that the latter operator is bounded from $L_{comp}^2(M)$ to $L_{loc}^2(M)$, and that $\mathfrak{B}^{-\frac{n(1-\rho)}{2}}$ is bounded from $H_{comp}^{-\frac{n(1-\rho)}{2}}(M)$ to $L_{loc}^2(M)$. Moreover, notice that because $\mathfrak{B}^{\frac{n(1-\rho)}{2}}$ is a positive operator \cite{stri}, the following inequality holds
\[
\mathfrak{B}^{-\frac{n(1-\rho)}{2}}(\phi u) \leq \mathfrak{B}^{-\frac{n(1-\rho)}{2}}(|\phi|\|u\|_{L^\infty(M)})= \|u\|_{L^\infty(M)}\mathfrak{B}^{-\frac{n(1-\rho)}{2}}(\phi),
\]
notice that we remove the absolute value because $\phi$ is positive. So we continue the estimation of \eqref{estimacionparaI} as follows:
\begin{align*}
    \eqref{estimacionparaI} & \leq \frac{1}{|B_r(0)|^{1/2}} \|a^1(X,D)\circ \mathfrak{B}^{\frac{n(1-\rho)}{2}}\|_{L^2(\bbR^n)} \|\mathfrak{B}^{-\frac{n(1-\rho)}{2}}(\phi u)\|_{L^2(\bbR^n)}\\
    &\leq \frac{C}{|B_r(0)|^{1/2}} \|a^1\|_{l_0; S_{\rho, \delta}^{-\frac{n(1-\rho)}{2}}} \|u\|_{L^\infty(M)} \|\mathfrak{B}^{-\frac{n(1-\rho)}{2}}(\phi)\|_{L^2(\bbR^n)}\\
    &\leq \frac{C}{|B_r(0)|^{1/2}} \|a^1\|_{l_0; S_{\rho, \delta}^{-\frac{n(1-\rho)}{2}}} \|u\|_{L^\infty(M)} \|\phi\|_{H^{-\frac{n(1-\rho)}{2}}(\bbR^n)},
\end{align*}
 for some integer $l_0$. So it only remains to estimate $\|\phi\|_{H^{-\frac{n(1-\rho)}{2}}(\bbR^n)}$. Using the properties of $\phi$ we obtain by direct computation the following bound:
\begin{align*}
\begin{split}
        \|\phi\|_{H^{-\frac{n(1-\rho)}{2}}(\bbR^n)}^2 &= \int_{\bbR^n} (1+|\nu|)^{-n(1-\rho)} |\widehat{\phi}(\nu)|^2 \, d\nu\\
    &\leq C \int_{|\nu|<r^{1/\rho}} \frac{1}{(1+|\nu|)^{n(1-\rho)}}\, d\nu \\
    & \leq C r^n.
\end{split}
\end{align*}
Coming back to the estimation of \eqref{estimacionparaI}, using property \eqref{normasDeLaParticion} we obtain:
\begin{align*}
    \eqref{estimacionparaI} & \leq \frac{Cr^{n/2}}{|B_r(0)|^{1/2}} \|a^1\|_{l_0; S_{\rho, \delta}^{-\frac{n(1-\rho)}{2}}} \|u\|_{L^\infty(M)}\\
    &\leq C \|a^1\|_{l_0; S_{\rho, \delta}^{-\frac{n(1-\rho)}{2}}} \|u\|_{L^\infty(M)}\\
    & \leq C \|a\|_{l_0; S_{\rho, \delta}^{-\frac{n(1-\rho)}{2}}} \|u\|_{L^\infty(M)}.
\end{align*}
Summarizing, we have proven that
\begin{equation}
    \label{3}
\frac{1}{|B_r(x_0)|} \int_{B_r(x_0)} |a^1(X, D)(\kappa^*(\phi) u)(y) |dy|^{-\kappa}| |dy| \leq C \|a\|_{l_0; S_{\rho, \delta}^{-\frac{n(1-\rho)}{2}}} \|u\|_{L^\infty(M)};
\end{equation}
for some $l_0$ and $0<r\leq \frac{r_0}{4}$, where $\kappa$ is the transformation to local coordinates used in the previous calculation. This completes the estimation for I.

We move on to the estimation of II. Let us write explicitly how II would look like if we calculate the commutator in normal coordinates, i.e we are looking for the symbol of the commutator. We have
\begin{align*}
    \operatorname{II} &= \phi(x) a^1(X,D)u(x) - a^1(X,D)(\phi u)(x)\\
    &= \frac{1}{(2\pi)^n}\left(\phi(x)\int_{\bbR^n} \int_{T_{x_0}^*M}  \Upsilon_x^{1-\kappa}(y)e^{i(x-y)\cdot \zeta} a^1\left(x, \zeta\right) u(y) d\zeta g(y)dy\right.\\
    &\left.-\int_{\bbR^n} \int_{T_{x_0}^*M}  \Upsilon_x^{1-\kappa}(y)e^{i(x-y)\cdot \zeta} a^1\left(x, \zeta\right) \phi(y)u(y) d\zeta g(y)dy\right)\\
    &= \frac{1}{(2\pi)^n}\int_{\bbR^n} \int_{T_{x_0}^*M}  \Upsilon_x^{1-\kappa}(y)e^{i(x-y)\cdot \zeta} a^1\left(x, \zeta\right)(\phi(x)- \phi(y))u(y) d\zeta g(y)dy. 
\end{align*}
We recognize that II is a pseudo-differential operator given by the amplitude 
$$b(y;x,\zeta)= a^1(x,\zeta)(\phi(x) - \phi(y)),$$ 
thus, as we did in the case of $a^0$, by Proposition \ref{AmpliToSym} we have that up to order 1 there exist a symbol $t^1(x,\zeta)$ such that 
\begin{align*}
    t^1(x,\zeta) &= \sum_{|\alpha|=1}D_\zeta^\alpha \nabla_y^\alpha b(y;x,\zeta) \bigg|_{y=x}\\
    &= \sum_{|\alpha|=1}D_\zeta^\alpha a^1(x,\zeta) \nabla_y^\alpha (\phi(x) - \phi(y)) \bigg|_{y=x}\\
    &= \sum_{|\alpha|=1}D_\zeta^\alpha a^1(x,\zeta) p_\alpha^1(x),
\end{align*}
where $p_\alpha^1$ is a bounded function coming from the Taylor expansion of $\phi$. Then 
\begin{align}
\label{controlTaylor2}
    |t^1(x,\zeta)| &= \left|\sum_{|\alpha|=1}D_\zeta^\alpha a^1(x,\zeta) p_\alpha^1(x) \right|\leq C \langle\zeta\rangle_x^{\frac{-n(1-\rho)}{2} - \rho}. 
\end{align}
Now, we use again a partition of unity such that 
\[
t^1(x,\zeta) = \sum_j \iota_j^1(x,\zeta),
\]
where $\iota_j^1$ is supported in $|\zeta|_x \sim 2^jr^{-1}$, here we are using the fact that the support of $a^1$ is $S_1$. From the condition on the supports and \eqref{controlTaylor2} we get for all $l$ that (similar as in the case of $a^0$)
\[
\|\iota_j^1\|_{l; S_{\rho, \delta}^{-\frac{n(1-\rho)}{2}}} \leq \frac{C_l}{2^{j\rho}}  \|a^1\|_{l; S_{\rho, \delta}^{-\frac{n(1-\rho)}{2}}}.
\]
Using Lemma \ref{lem2} on the $\iota_j^1$'s, $l>n/2$ and property \eqref{normasDeLaParticion}, we get 
\begin{align}
\label{4}
\begin{split}
        \|\operatorname{II}\|_{L^\infty(M)} &\leq \sum_j \|t_j(X,D)u\|_{L^\infty(M)}\\
    &\leq C\sum_j  \frac{1}{2^{j\rho}} \|a^1\|_{l; S_{\rho, \delta}^{-\frac{n(1-\rho)}{2}}} \|u\|_{L^\infty(M)}\\
    &\leq C \|a^1\|_{l; S_{\rho, \delta}^{-\frac{n(1-\rho)}{2}}} \|u\|_{L^\infty(M)}\\
    &\leq C \|a\|_{l; S_{\rho, \delta}^{-\frac{n(1-\rho)}{2}}} \|u\|_{L^\infty(M)}.
\end{split}
\end{align}
This completes the estimation for II and finishes the process for $a^1$. 
From \eqref{1}, \eqref{22}, \eqref{23}, \eqref{2}, \eqref{3} and \eqref{4}, by taking the corresponding suprema as in \eqref{defBMO} we conclude that
\[
\|a(X, D)_{loc}u\|_{\operatorname{BMO}(M)}\leq C \|a\|_{l; S_{\rho, \delta}^{-\frac{n(1-\rho)}{2}}}\|u\|_{L^\infty(M)},
\]
completing the proof.
\end{proof}
Due to the abstract duality result of Proposition \ref{dualidadbmo} and Theorem \ref{adjunto} we immediately obtain the following corollary:
\begin{cor}
    Let at least one of the conditions (1)-(3) of Theorem \ref{producto} be fulfilled. Let $a\in S_{\rho, \delta}^{-\frac{n(1-\rho)}{2}}(\nabla)$. Then $a(X, D)_{loc}$ is locally bounded from $H^1(M, \Omega^\kappa)$ to $L^1(M,\Omega^\kappa)$. 
\end{cor}
Now we interpolate, utilizing Proposition \ref{interpolacion}, between the result of Theorem \ref{bmo} and the $L^2$ boundedness. 
\begin{thm}
\label{lpParaOrdenFijo}
Let $a\in S_{\rho, \delta}^{-\frac{n(1-\rho)}{2}}(\nabla)$ with $0\leq\delta<\rho\leq 1$. Suppose that at least one of the following conditions is fulfilled:
\begin{enumerate}
    \item $\rho>\frac{1}{2}$;
    \item the connection $\nabla$ is symmetric and $\rho>\frac{1}{3}$;
    \item the connection $\nabla$ is flat.
\end{enumerate}  Then $a(X, D)_{loc}$ is locally bounded from $L^p(M,\Omega^\kappa)$ to $L^p(M,\Omega^\kappa)$ for $1<p<\infty$. Moreover, if $a\in S_{\rho,\delta}^{-\frac{n(1-\rho)}{2}}(\nabla)_{glo}$ then $a(X,D)$ is globally bounded from $L^p(M,\Omega^\kappa)$ to $L^p(M,\Omega^\kappa)$. 
\end{thm}
\begin{proof}
We are in a position to apply Theorem \ref{bmo} so that $A=a(X, D)$ is bounded from $L^\infty(M)$ to $\operatorname{BMO}(M)$. Since the order of $A$ is negative it belongs to the class $\Psi_{\rho, \delta}^{0}(\Omega^\kappa,\nabla)$ as well, thus by Theorem \ref{L2bound} it is bounded from $L^2(M)$ to $L^2(M)$. By Proposition \ref{interpolacion} we get the boundedness from $L^p(M)$ to $L^p(M)$ for $2\leq p< \infty$. On the other hand, Theorem \ref{adjunto} together with a similar argument as before give us the boundedness of $A^*$ from $L^p(M)$ to $L^p(M)$ for $2\leq p< \infty$. Therefore, by duality we obtain the boundedness of $A$ on the missing interval $1< p\leq 2$. Finally, the global boundedness follows from this and Lemma \ref{lemGlob}.
\end{proof}
Previous theorem guarantees $L^p$-boundedness for operators of a fixed order, we use the complex interpolation from Stein \cite{Stein} to extend the range of orders and finally get the desired result:
\begin{thm}
\label{lpFinal}
Let $\theta\in [0,\frac{n(1-\rho)}{2})$ and let $a\in S_{\rho, \delta}^{-\theta}(\nabla)$  with $0\leq\delta<\rho\leq 1$. Suppose that at least one of the following conditions is fulfilled:
\begin{enumerate}
    \item $\rho>\frac{1}{2}$;
    \item the connection $\nabla$ is symmetric and $\rho>\frac{1}{3}$;
    \item the connection $\nabla$ is flat.
\end{enumerate}
Then $a(X, D)_{loc}$ is locally bounded from $L^p(M,\Omega^\kappa)$ to $L^p(M,\Omega^\kappa)$ for 
$$\left|\frac{1}{p}-\frac{1}{2}\right|\leq \frac{\theta}{n(1-\rho)}.$$
 Moreover, if $a\in S_{\rho,\delta}^{-\frac{n(1-\rho)}{2}}(\nabla)_{glo}$ then $a(X,D)$ is globally bounded from $L^p(M,\Omega^\kappa)$ to $L^p(M,\Omega^\kappa)$ for the same range of values of $p$. 
\end{thm}
\begin{proof}
We are going to use a complex interpolation argument, meaning that we are going to define a family of operators parametrized by a complex variable lying in a certain strip and proof the boundedness of its symbols. Define the complex parametrized family of symbols 
$$\mathfrak{a}_z(y,\eta) := e^{z^2}a(x,\eta)\langle\eta\rangle_{y}^{\theta + \frac{n(1-\rho)}{2}(z-1)}$$
for $z$ belonging to the strip $\Bar{S} =\{z : 0\leq \operatorname{Re}z\leq 1\}$ and $a\in S_{\rho, \delta}^{-\theta}(\nabla)$. Set $\{\mathfrak{A}_z\}_{0\leq \operatorname{Re}z\leq 1}$ the corresponding family of pseudo-differential operators. Note that for a fixed $z$ each of these operators is of order $0$. 

Let  $z= \varepsilon + i\mu$, $\varepsilon\in[0,1]$ and $\mu\in\bbR$. Following Example \ref{cotaDerivadademetrica} we compute the following seminorm for $l>n/2$: 
\begin{align*}
    \|\mathfrak{a}_z\|_{l; S_{\rho, \delta}^{0}} &= \sup_{|\alpha|+ q \leq l, (y,\eta)} \frac{\left|\partial_{\eta}^{\alpha} \nabla_{i_{1}} \ldots \nabla_{i_{q}} (e^{z^2}a(x,\eta)\langle\eta\rangle_{y}^{\theta + \frac{n(1-\rho)}{2}(z-1)})\right|}{\langle\eta\rangle_{y}^{\delta q-\rho|\alpha|}}\\
    &\leq \sup_{|\alpha|+ q \leq l, (y,\eta)} \frac{e^{\varepsilon^2 - \mu^2} \left|p\left(\theta + \frac{n(1-\rho)}{2}(z-1)\right)\right| \langle\eta\rangle_{y}^{\theta + \frac{n(1-\rho)}{2}(\varepsilon-1) - l}\left|\partial_{\eta}^{\alpha} \nabla_{i_{1}} \ldots \nabla_{i_{q}} a(x,\eta)\right|}{\langle\eta\rangle_{y}^{\delta q-\rho|\alpha|}}\\
    &\leq e^{\varepsilon^2 - \mu^2} \left|p\left(\theta + \frac{n(1-\rho)}{2}(z-1)\right)\right| \sup_{|\alpha|+ q \leq l, (y,\eta)} \frac{\left|\partial_{\eta}^{\alpha} \nabla_{i_{1}} \ldots \nabla_{i_{q}} a(x,\eta)\right|}{\langle\eta\rangle_{y}^{-\theta + \delta q-\rho|\alpha|}}\\
    &\leq e^{\varepsilon^2 - \mu^2} \left|p\left(\theta + \frac{n(1-\rho)}{2}(z-1)\right)\right| \|a\|_{l; S_{\rho, \delta}^{-\theta}},
\end{align*}
where $p$ is a polynomial of degree $l$. From the conditions imposed in $z$ and the fact that $l>n/2$ one can bound the product of the first two terms as follows: 
\begin{align*}
    e^{\varepsilon^2 - \mu^2} \left|p\left(\theta + \frac{n(1-\rho)}{2}(z-1)\right)\right| &\leq  C_l e^{1 - \mu^2} \left|\theta + \frac{n(1-\rho)}{2}(z-1)\right|^{l}\\
    &\leq C_l e^{ - \mu^2}|z|^l\\
    &\leq C_l \frac{(1 + |\mu|^2)^{l/2}}{e^{\mu^2}}\\
    &\leq C_l. 
\end{align*}
Therefore, for an appropriate choose of $l$ we can control certain seminorms of the complex family of symbols by the ones of the initial symbol, specifically we obtained the following inequality
\begin{equation}
    \|\mathfrak{a}_z\|_{l; S_{\rho, \delta}^{0}} \leq C_l \|a\|_{l; S_{\rho, \delta}^{-\theta}}. 
\end{equation}
By Theorem \ref{L2bound} there exist $C>0$ and an integer $l_0$ such that 
\begin{equation}
    \|\mathfrak{A}_zu\|_{L^2(M)}\leq C \|\mathfrak{a}_z\|_{l_0; S_{\rho, \delta}^{0}} \|u\|_{L^2(M)}. 
\end{equation}
Putting together the last two inequalities we get 
\begin{equation}
\label{cotaL2Compl}
    \|\mathfrak{A}_zu\|_{L^2(M)}\leq C \|a\|_{l_0; S_{\rho, \delta}^{-\theta}} \|u\|_{L^2(M)}.
\end{equation}
Moreover, notice that the family of operators $\{\mathfrak{A}_z\}_{0\leq \operatorname{Re}z\leq 1}$ is analytic in the strip $S =\{z : 0< \operatorname{Re}z< 1\}$ (just exponential dependence on the symbol) and continuous on the closure $\Bar{S}$. These observations and inequality \eqref{cotaL2Compl} imply that the family $\{\mathfrak{A}_z\}_{0\leq \operatorname{Re}z\leq 1}$ defines an analytic family of operators uniformly bounded from $L^2(M)$ to $L^2(M)$, so the first condition of the complex interpolation procedure is fulfilled. We proceed with the other conditions. On one hand, using \eqref{cotaL2Compl} we have for the right border of the strip that
\[
\sup_{-\infty<\mu<+\infty} \|\mathfrak{A}_{1+i\mu}u\|_{L^2}\leq C \|a\|_{l_0; S_{\rho, \delta}^{-\theta}} \|u\|_{L^2}, \quad u\in L^2(M),
\]
where $C$ is independent of $u$. On the other hand, for the left border of the strip the family's symbols take the form 
\[
\mathfrak{a}_{i\mu}(y,\eta) := e^{-\mu^2}a(x,\eta)\langle\eta\rangle_{y}^{\theta} \langle\eta\rangle_{y}^{-\frac{n(1-\rho)}{2}} \langle\eta\rangle_{y}^{i\frac{\mu n(1-\rho)}{2}},
\]
and performing a similar estimation as before we can prove that $\mathfrak{a}_{i\mu}\in S_{\rho, \delta}^{-\frac{n(1-\rho)}{2}}(\nabla)$, moreover 
\[
\|\mathfrak{a}_{i\mu}\|_{l; S_{\rho, \delta}^{-\frac{n(1-\rho)}{2}}} \leq C \|a\|_{l; S_{\rho, \delta}^{-\theta}}
\]
for $l>n/2$ and $C>0$ independent of $\mu$. Thus, we are in position to apply Theorem \ref{bmo}, which give us 
\[
\|\mathfrak{A}_{i\mu}u\|_{\operatorname{BMO}(M)}\leq C \|a\|_{l; S_{\rho, \delta}^{-\theta}}\|u\|_{L^\infty(M)},
\]
and with this inequality we complete the conditions to apply the complex interpolation. From this abstract procedure we obtain the following $L^p-L^p$ boundedness for operators $\mathfrak{A}_{\varepsilon}$ (sub-family of operators with imaginary part equal to zero):
\begin{equation}
\label{lpPruebaCom}
    \|\mathfrak{A}_{\varepsilon}u\|_{L^p(M)}\leq C_p\|a\|_{l; S_{\rho, \delta}^{-\theta}}\|u\|_{L^p(M)},
\end{equation}
where $p=\frac{2}{\varepsilon}$ and $0<\varepsilon\leq1$ (each element of the family is bounded in a different $L^p$ space). Remember that these operators $\mathfrak{A}_{\varepsilon}$ have symbols 
\[
\mathfrak{a}_\varepsilon(y,\eta) := e^{\epsilon^2}a(x,\eta)\langle\eta\rangle_{y}^{\theta + \frac{n(1-\rho)}{2}(\varepsilon-1)}, 
\]
and notice that we can get $L^p$ boundedness for the initial operator $a(X,D)$ of order $\theta\in [0,\frac{n(1-\rho)}{2})$ since obviously there exists an $0<\varepsilon_0\leq 1$ such that 
\[
\theta = \frac{n(1-\rho)}{2}(1-\varepsilon_0).
\]
Thus it follows from \eqref{lpPruebaCom} that for $p=\frac{2}{\varepsilon_0}$ we have
\begin{equation}
\label{estimativoLpParainterpolar}
    \|a(X,D)u\|_{L^p}\leq C_p\|a\|_{l; S_{\rho, \delta}^{-\theta}}\|u\|_{L^p}.
\end{equation}
Observe that we can rewrite $\theta$ in terms of $p$
\[
\theta = \frac{n(1-\rho)}{2}(1-\varepsilon_0) = n(1-\rho)\left(\frac{1}{2}-\frac{1}{p}\right).
\]
By interpolation of the $L^2$ estimate coming from Theorem \ref{L2bound} and the $L^{2/\varepsilon_0}$ estimate \eqref{estimativoLpParainterpolar} we obtain the $L^p$-boundedness for values of $p$ satisfying 
\[
\frac{1}{2}-\frac{1}{p}\leq \frac{\theta}{n(1-\rho)}. 
\]
As in the proof of Theorem \ref{lpParaOrdenFijo} a duality argument give us the desired result on the remaining interval
\[
\frac{1}{p}-\frac{1}{2}\leq \frac{\theta}{n(1-\rho)}.
\]
Therefore, we have proved that $a(X,D): L^p(M)\to L^p(M)$ is bounded for $p$ verifying  
$$\left|\frac{1}{p}-\frac{1}{2}\right|\leq \frac{\theta}{n(1-\rho)},$$
completing the proof. 
\end{proof}
\begin{rem}
This result is sharp in the following sense: Let at least one of the conditions (1)-(3) of Theorem \ref{producto} be fulfilled. Let $\theta\in [0,\frac{n(1-\rho)}{2})$ and suppose $p$ satisfies 
$$\left|\frac{1}{p}-\frac{1}{2}\right|> \frac{\theta}{n(1-\rho)}.$$ 
Let $x_0\in \bbS^1$ be fixed, then the symbol 
$$a(x,\zeta) = a_{\rho\theta}(\zeta) = \frac{e^{i|\eta|_{x_0}^{1-\rho}}}{1 + |\eta|_{x_0}^\theta}\in S_{\rho,0}^{-\theta}(\nabla)$$
produces an operator $a_{\alpha\beta}(D)$ which is unbounded from $L^p(\bbS^1)$ to $L^p(\bbS^1)$. This is just the classical counterexample of Hardy-Littlewood-Hirschman-Wainger \cite{Counter} in the context of these classes. 
\end{rem}

\begin{rem}
\label{rem1}
Under the first condition $\rho>1/2$ we recover the classical result on any manifold; under the second condition the connection $\nabla$ is symmetric and $\rho>\frac{1}{3}$ we extend the result to any manifold since we can always find a symmetric connection (not necessarily the Levi-Civita connection since we do not assume metricity); under the third condition the connection $\nabla$ is flat we extend further the classical result, but not to any manifold $M$ since flatness is a major constraint. Nevertheless, since we do not assume metricity we do not have only finite number of  examples per dimension that satisfy that condition. In the next figure we compare the different values of $p$ for three different values of $\rho$; $\rho>1/2$ which gives $\frac{n}{4}$ as upper bound for $\theta$, $\rho>1/3$ which gives $\frac{7n}{20}$ as upper bound for $\theta$, and $\rho>0$ which gives $\frac{n}{2}$ as upper bound for $\theta$:
\begin{tabular}{cc}
   \resizebox{0.39\textwidth}{!}{%
\begin{circuitikz}
\tikzstyle{every node}=[font=\small]
\draw [->, >=Stealth] (0,0) .. controls (2,0) and (4,0) .. (6.3,0);
\draw [->, >=Stealth] (0,0) .. controls (0,5) and (0,5) .. (0,6.3);
\node [font=\normalsize] at (-1,6) {$\frac{1}{p}$};
\node [font=\small] at (6,-0.5) {$\theta$};
\node [font=\normalsize] at (-1,5) {$1$};
\node [font=\small] at (5,-0.5) {$\frac{n}{2}$};
\node [font=\small] at (-1,2.5) {$\frac{1}{2}$};
\node [font=\small] at (2.5,-0.5) {$\frac{n}{4}$};
\node [font=\small] at (3.5,-0.5) {$\frac{7n}{20}$};
\node [font=\small] at (-1,-0.5) {$0$};

\draw [color=red, dashed](2.5,0) to[short] (2.5,5);
\draw [](0,5) to[short] (5,5);
\draw [color=red](0,2.5) to[short] (2.5,5);
\draw [color=red](0,2.5) to[short] (2.5,0);
\fill[color={rgb,255:red,255; green,204; blue,204}]  (0,2.5) -- (2.5,5) -- (2.5,0) -- cycle;
[
\end{circuitikz}
} & \resizebox{0.39\textwidth}{!}{%
\begin{circuitikz}
\tikzstyle{every node}=[font=\small]
\draw [->, >=Stealth] (0,0) .. controls (2,0) and (4,0) .. (6.3,0);
\draw [->, >=Stealth] (0,0) .. controls (0,5) and (0,5) .. (0,6.3);
\node [font=\normalsize] at (-1,6) {$\frac{1}{p}$};
\node [font=\small] at (6,-0.5) {$\theta$};
\node [font=\normalsize] at (-1,5) {$1$};
\node [font=\small] at (5,-0.5) {$\frac{n}{2}$};
\node [font=\small] at (-1,2.5) {$\frac{1}{2}$};
\node [font=\small] at (2.5,-0.5) {$\frac{n}{4}$};
\node [font=\small] at (3.5,-0.5) {$\frac{7n}{20}$};
\node [font=\small] at (-1,-0.5) {$0$};

\draw [color=green, dashed](3.5,0) to[short] (3.5,5);
\draw [](0,5) to[short] (5,5);
\draw [color=green](0,2.5) to[short] (3.5,5);
\draw [color=green](0,2.5) to[short] (3.5,0);
\fill[color={rgb,255:red,204; green,255; blue,204}]  (0,2.5) -- (3.5,5) -- (3.5,0) -- cycle;
[
\end{circuitikz}}\\
\end{tabular}

\begin{figure}[!ht]
\centering
\resizebox{0.39\textwidth}{!}{%
\begin{circuitikz}
\tikzstyle{every node}=[font=\small]
\draw [->, >=Stealth] (0,0) .. controls (2,0) and (4,0) .. (6.3,0);
\draw [->, >=Stealth] (0,0) .. controls (0,5) and (0,5) .. (0,6.3);
\node [font=\normalsize] at (-1,6) {$\frac{1}{p}$};
\node [font=\small] at (6,-0.5) {$\theta$};
\node [font=\normalsize] at (-1,5) {$1$};
\node [font=\small] at (5,-0.5) {$\frac{n}{2}$};
\node [font=\small] at (-1,2.5) {$\frac{1}{2}$};
\node [font=\small] at (2.5,-0.5) {$\frac{n}{4}$};
\node [font=\small] at (3.5,-0.5) {$\frac{7n}{20}$};
\node [font=\small] at (-1,-0.5) {$0$};

\draw [color=blue, dashed](5,0) to[short] (5,5);
\draw [](0,5) to[short] (5,5);
\draw [color=blue](0,2.5) to[short] (5,5);
\draw [color=blue](0,2.5) to[short] (5,0);
\fill[color={rgb,255:red,204; green,204; blue,255}]  (0,2.5) -- (5,5) -- (5,0) -- cycle;
[
\end{circuitikz}
}%
\label{grafica}
\caption{In red the possible values of $p$ when $\rho>1/2$, in green the possible values of $p$ when $\rho>1/3$ and $\nabla$ is symmetric, in blue the possible values of $p$ when $\rho>0$ and $\nabla$ is flat. }
\end{figure}
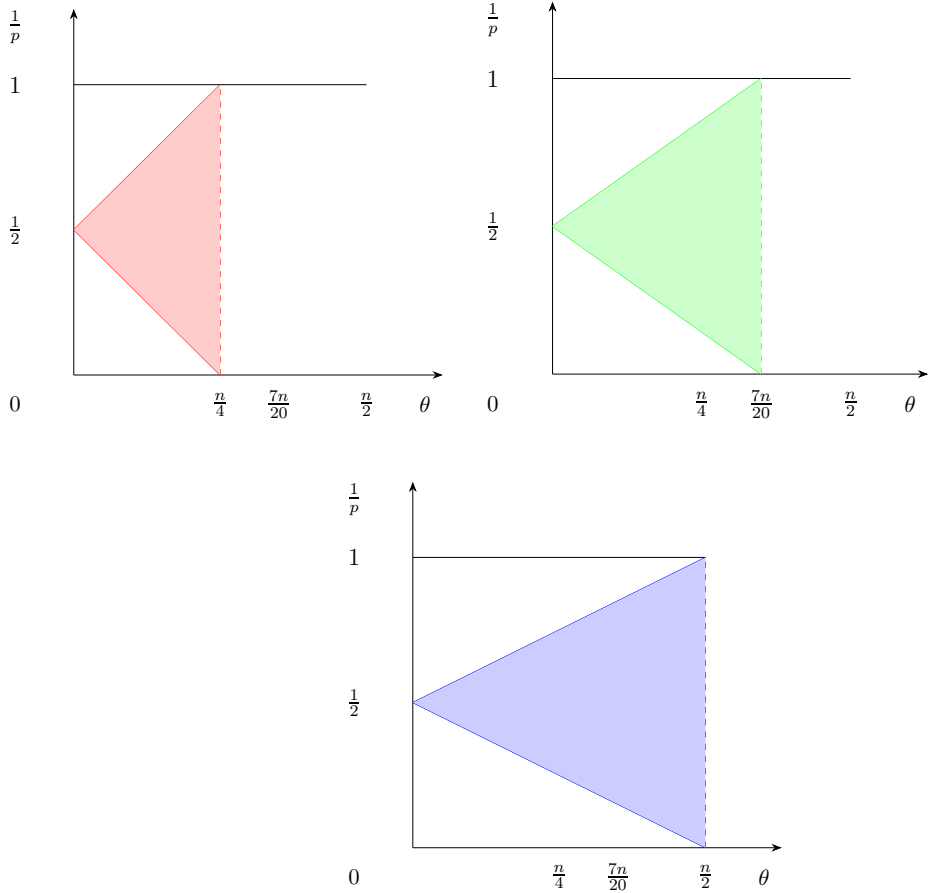
Moreover, notice that Theorem \ref{lpParaOrdenFijo} is giving a result in the border case $\theta=\frac{n(1-\rho)}{2}$ for $1<p<\infty$. 
\end{rem}
\subsection{Further consequences}
In this subsection we apply our main result Theorem \ref{lpFinal} combined with standard arguments regarding Bessel potentials to get estimates in several different function spaces. 
\subsubsection{Sobolev and Besov boundedness}
First of all, we can have boundedness on Sobolev spaces $L_s^p(M)$:
\begin{cor}
\label{cotaSobolev}
Let at least one of the conditions (1)-(3) of Theorem \ref{lpFinal} be fulfilled. Let $\theta\in\bbR$ and $a\in S_{\rho, \delta}^{-\theta}(\nabla)$. Let $s_1,s_2\in\bbR$ such that
    \begin{equation}
        -\frac{n(1-\rho)}{2}<-\theta-s_1+s_2\leq 0.
    \end{equation}
    Then $a(X, D)_{loc}$ is locally bounded from $L_{s_1}^p(M,\Omega^\kappa)$ to $L_{s_2}^p(M,\Omega^\kappa)$ for 
$$\left|\frac{1}{p}-\frac{1}{2}\right|\leq \frac{\theta}{n(1-\rho)}.$$ 
Moreover, if $a\in S_{\rho, \delta}^{-\theta}(\nabla)_{glo}$ then $a(X,D)$ is globally bounded from $L_{s_1}^p(M,\Omega^\kappa)$ to $L_{s_2}^p(M,\Omega^\kappa)$ for the same range of values of $p$. 
\end{cor}
\begin{proof}
For $u\in L_{s_2}^p(M)$ we want to prove the estimate 
\[
\|a(X,D) u\|_{L_{s_1}^p(M)} \leq C \|u\|_{L_{s_2}^p(M)},
\]
but due to Bessel potentials properties this inequality is equivalent to 
\[
\|\mathfrak{B}^{s_2} a(X,D) \mathfrak{B}^{-s_1} u\|_{L^p} \leq C \|u\|_{L^p}. 
\]
By \eqref{besselEnSafarov} and Theorem \ref{producto}, $\mathfrak{B}^{s_2} a(X,D) \mathfrak{B}^{-s_1}\in \Psi_{\rho, \delta}^{-\theta-s_1+s_2}\left(\Omega^{\kappa}, \nabla\right)$, which by hypothesis such order is in the valid range of Theorem \ref{lpFinal}, hence the result follows. 
\end{proof}
\begin{rem}
    Notice that if $\theta\in [0,\frac{n(1-\rho)}{2})$ and we are given any $s\in\bbR$, we can set $s_1=s_2=s$ and the previous result will always hold. 
\end{rem}
Moreover, we can have boundedness on Besov spaces $B_{p,q}^s(M)$:
\begin{cor}
   Let at least one of the conditions (1)-(3) of Theorem \ref{lpFinal} be fulfilled. Let $\theta\in [0,\frac{n(1-\rho)}{2}), s\in \bbR, 0<q\leq \infty$, and $a\in S_{\rho, \delta}^{-\theta}(\nabla)$. Then $a(X, D)_{loc}$ is locally bounded from $B_{p,q}^s(M,\Omega^\kappa)$ to $B_{p,q}^s(M,\Omega^\kappa)$ for 
$$\left|\frac{1}{p}-\frac{1}{2}\right|\leq \frac{\theta}{n(1-\rho)}.$$
Moreover, if $a\in S_{\rho, \delta}^{-\theta}(\nabla)_{glo}$ then $a(X,D)$ is globally bounded from $B_{p,q}^s(M,\Omega^\kappa)$ to $B_{p,q}^s(M,\Omega^\kappa)$ for the same range of values of $p$. 
\end{cor}
\begin{proof}
Let $0<\Theta<1$ and $-\infty< s_0 <s_1 < \infty$ such that 
\[
s = (1-\Theta)s_0+\Theta s_1.
\]
Thus by Corollary \ref{cotaSobolev} we have that $a(X,D)$ is bounded from $L_{s_0}^p(M)$ to $L_{s_0}^p(M)$, and simultaneously from $L_{s_1}^p(M)$ to $L_{s_1}^p(M)$. Now, it follows from real interpolation that the operator
\[
a(X,D): (L_{s_0}^p(M), L_{s_1}^p(M))_{\Theta, q}\to (L_{s_0}^p(M), L_{s_1}^p(M))_{\Theta, q}
\]
is bounded for $0<q\leq \infty$. Hence, Corollary \ref{interoplacionSobolevBesov} gives us the desired result. 
\end{proof}

\subsubsection{\texorpdfstring{$L^p-L^q$}{L}-estimates}
Finally, we conclude this work by proving estimations of the type
\begin{equation}\label{pq}
    \|Au\|_{L^q(M)}\leq C \|u\|_{L^q(M)}, 
\end{equation}
i.e. $L^p-L^q$ estimates, again using techniques involving Bessel potentials. 
\begin{cor}\label{corlplp}
Let at least one of the conditions (1)-(3) of Theorem \ref{lpFinal} be fulfilled. Let $1<p\leq2\leq q<\infty$, $\theta\in\bbR$, and $a\in S_{\rho, \delta}^{-\theta}(\nabla)$. Then $a(X, D)_{loc}$ is locally bounded from $L^p(M,\Omega^\kappa)$ to $L^q(M,\Omega^\kappa)$ if
\begin{equation}\label{LPLQ1}
    n\left(\frac{1}{p}-\frac{1}{q}\right)\leq -\theta. 
\end{equation}
Moreover, if $a\in S_{\rho, \delta}^{-\theta}(\nabla)_{glo}$ then $a(X,D)$ is globally bounded from $L^p(M,\Omega^\kappa)$ to $L^q(M,\Omega^\kappa)$ for the same range of values of $p$ and $q$. 
\end{cor}
\begin{proof}
   Let $\theta$ satisfy condition \eqref{LPLQ1}. Take $\theta_1, \theta_2\in\bbR$ such that $-\theta=-\theta_1-\theta_2$, $-\theta_1\geq n(\frac{1}{p}-\frac{1}{2})$ and $-\theta_2\geq n(\frac{1}{2}-\frac{1}{q})$. Thus, we can rewrite $A$ as follows 
    \[
    A = \mathfrak{B}^{\theta_2}\underbrace{\mathfrak{B}^{-\theta_2}A\mathfrak{B}^{-\theta_1}}_{\mathfrak{Z}}\mathfrak{B}^{\theta_1},  
    \]
    where by \eqref{besselEnSafarov} and Theorem \ref{producto}, $\mathfrak{Z}\in\Psi_{\rho, \delta}^{0}\left(\Omega^{\kappa}, \nabla\right)$. As a consequence of Theorems \ref{L2bound} and \ref{sobolev} we have that the previous operators are bounded on the following spaces 
    \[
    L^p(M)\xrightarrow{\mathfrak{B}^{\theta_1}} L^2(M)\xrightarrow{\mathfrak{Z}} L^2(M) \xrightarrow{\mathfrak{B}^{\theta_2}}L^q(M),
    \]
    so that $A$ is bounded from $L^p(M)$ to $L^q(M)$. 
\end{proof}
\begin{rem}\label{rem2}
    Notice that the values of $p$ and $q$ which one can have in this result are independent of $\rho$ and $\delta$. Also notice that we can not have $\theta$ to be strictly positive (the case of differential operators) because then it is impossible to satisfy \eqref{LPLQ1}, so this Corollary is useful for pseudo-differential operators of non-positive order. If $\theta=0$, then we only have a $L^2-L^2$ estimate. Moreover, if $-\theta\geq n$ then Condition \eqref{LPLQ1} is not affecting at all the values of $p$ and $q$. In the following Figure we illustrate how one can visualize Corollary \ref{corlplp}. 
    \begin{figure}[!ht]
\centering
\resizebox{0.4\textwidth}{!}{%
\begin{circuitikz}
\tikzstyle{every node}=[font=\small]
\draw [->, >=Stealth] (0,0) .. controls (2,0) and (4,0) .. (6.3,0);
\draw [->, >=Stealth] (0,0) .. controls (0,5) and (0,5) .. (0,6.3);
\node [font=\normalsize] at (-1,6) {$\frac{1}{q}$};
\node [font=\small] at (6,-0.5) {$\frac{1}{p}$};
\node [font=\normalsize] at (-1,5) {$1$};
\node [font=\small] at (5,-0.5) {$1$};
\node [font=\small] at (-1,2.5) {$\frac{1}{2}$};
\node [font=\small] at (-1,0.5) {$\frac{1}{10}$};
\node [font=\small] at (2.5,-0.5) {$\frac{1}{2}$};
\node [font=\small] at (4.5,-0.5) {$\frac{9}{10}$};
\node [font=\small, color=red]  at (1.9,-0.5) {$-\frac{\theta}{n}$};
\node [font=\small] at (-1,-0.5) {$0$};
\node [circle,fill,inner sep=1pt, color=red] at (2.5,2.5) {};

\draw [](5,2.5) to[short] (5,5);
\draw [color=red, dashed](2.5,0) to[short] (2.5,2.5);
\draw [color=red, dashed](5,0) to[short] (5,2.5);
\draw [color=red, dashed](2.5,2.5) to[short] (5,2.5);
\draw [color=red, dashed](2.5,0) to[short] (5,0);
\draw [](0,5) to[short] (5,5);
\draw[color=red] (2,0) to (5,3);
\fill[color={rgb,255:red,255; green,204; blue,204}]  (2.5,0.5) -- (2.5,2.5) -- (4.5,2.5) --  cycle;
[
\end{circuitikz}
}%
\label{8}
\caption{The coloured red area (triangle) represents the values of $p$ and $q$ one can insert in inequality \eqref{pq} when $0<-\theta<n$ and $1<p\leq2\leq q<\infty$.}
\end{figure}
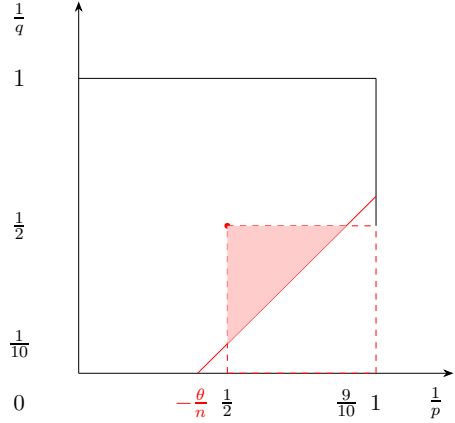
\end{rem}

\begin{cor}\label{corlplq2}
    Let at least one of the conditions (1)-(3) of Theorem \ref{lpFinal} be fulfilled. Let $\theta\in\bbR$ and $a\in S_{\rho, \delta}^{-\theta}(\nabla)$. Then $a(X, D)_{loc}$ is bounded from $L^p(M,\Omega^\kappa)$ to $L^q(M,\Omega^\kappa)$ if
    \begin{enumerate}
        \item $1<p\leq q \leq 2$ and 
        \begin{equation}
        \label{condLpLqUltima}
           n\left(\frac{1}{p}-\frac{1}{q} + (1-\rho)\left(\frac{1}{q}-\frac{1}{2}\right)\right)\leq -\theta;
        \end{equation}
        \item or if $2\leq p\leq q <\infty $ and 
        \begin{equation}
           n\left(\frac{1}{p}-\frac{1}{q} + (1-\rho)\left(\frac{1}{2}-\frac{1}{p}\right)\right)\leq -\theta. 
        \end{equation}
    \end{enumerate}
    Moreover, in both cases if $a\in S_{\rho, \delta}^{-\theta}(\nabla)_{glo}$ then $a(X,D)$ is globally bounded from $L^p(M,\Omega^\kappa)$ to $L^q(M,\Omega^\kappa)$ for the same range of values of $p$ and $q$. 
\end{cor}
\begin{proof} The proofs of $(1)$ and $(2)$ are essentially the same, so we only do the one of $(1)$. Fix $\theta_0=-n(\frac{1}{p}-\frac{1}{q})$, so that $\mathfrak{B}^{\theta_0}$ is bounded from $L^p(M)$ to $L^q(M)$ by Theorem \ref{sobolev}. We rewrite $A$ as follows
    \[
    A=(A\mathfrak{B}^{-\theta_0})\mathfrak{B}^{\theta_0},
    \]
    where by \eqref{besselEnSafarov} and Theorem \ref{producto}, $A\mathfrak{B}^{-\theta_0}\in \Psi_{\rho, \delta}^{\theta-\theta_0}\left(\Omega^{\kappa}, \nabla\right)$. From condition \eqref{condLpLqUltima} and definition of $\theta_0$ we get that 
    \[
    (1-\rho)\left(\frac{1}{q}-\frac{1}{2}\right)\leq -(\theta-\theta_0). 
    \]
    Hence, by Theorem \ref{lpFinal}, $A\mathfrak{B}^{-\theta_0}$ is bounded from $L^q(M)$ to $L^q(M)$ and the result follows taking into account Remark \ref{NormaBessel}.  
\end{proof}
\begin{rem}\label{rem3}
    Unlike in Corollary \ref{corlplp}, in Corollary \ref{corlplq2} (1) and (2) we do have an important dependence on $\rho$. We are going to comment in detail the case $(1)$, but for case (2) the analysis is analogous. Again, the possibility of taking $\theta$ positive is discarded because then it will be impossible to fulfill Condition \eqref{condLpLqUltima}. Let us rewrite Condition \eqref{condLpLqUltima} as follows: 
    \begin{equation}\label{linea}
        \frac{1}{\rho p}+\frac{1}{\rho}\left(\frac{\theta}{n}-\frac{1-\rho}{2}\right)\leq \frac{1}{q}, 
    \end{equation}
    where we recognize that this is a line with slope $\frac{1}{\rho}$ and cutting the x-axis at $\frac{1}{\rho}\left(\frac{1-\rho}{2}-\frac{\theta}{n}\right)$. 
    If $\theta=0$, we have two cases: for $\rho=1$ the values of $p$ and $q$ are exactly the line $1<p=q\leq 2$, for $0<\rho<1$ we only have the $L^2-L^2$ estimate. If $-\theta\geq n$, then this condition does not affect at all the triangle $\frac{1}{2}\leq\frac{1}{q}\leq\frac{1}{p}<1$. If $0<-\theta<n$, then regarding $\rho$ as variable one can think as follows: the parameterized family of lines \eqref{linea} is rotating the line with $\rho=1$
    \[
    \frac{1}{ p}+ \frac{\theta}{n}\leq \frac{1}{q},
    \]
    into the vertical line ($\rho=0$)
    \[
    \frac{1}{p}\leq -\frac{\theta}{n}+\frac{1}{2},
    \]
    while fixing the point $(-\frac{\theta}{n}+\frac{1}{2}, \frac{1}{2})$. The latter observation implies that in this case, the condition is only relevant when $0<-\frac{\theta}{n}<\frac{1}{2}$. Hence, the enclosed area between the family of lines \eqref{linea} and the triangle $\frac{1}{2}\leq\frac{1}{q}\leq\frac{1}{p}<1$ is changing as we change $\rho$, the bigger the value of $\rho$ the bigger the enclosed area. This phenomena is shown in the following Figure:
    \begin{figure}[!ht]
        \begin{tabular}{cc}
        \resizebox{0.39\textwidth}{!}{%
\begin{circuitikz}
\tikzstyle{every node}=[font=\small]
\draw [->, >=Stealth] (0,0) .. controls (2,0) and (4,0) .. (6.3,0);
\draw [->, >=Stealth] (0,0) .. controls (0,5) and (0,5) .. (0,6.3);
\node [font=\normalsize] at (-1,6) {$\frac{1}{q}$};
\node [font=\small] at (6,-0.5) {$\frac{1}{p}$};
\node [font=\normalsize] at (-1,5) {$1$};
\node [font=\small] at (5,-0.5) {$1$};
\node [font=\small] at (-1,2.5) {$\frac{1}{2}$};
\node [font=\small] at (2.5,-0.5) {$\frac{1}{2}$};
\node [font=\small] at (4,-0.5) {$-\frac{\theta}{n}+\frac{1}{2}$};

\node [font=\small] at (-1,-0.5) {$0$};
\node [circle,fill,inner sep=1pt, color=black] at (2.5,2.5) {};


\draw [color=black, dashed](5,2.5) to[short] (5,5);
\draw [color=black, dashed](2.5,2.5) to[short] (5,2.5);
\draw [](0,5) to[short] (5,5);
\draw [](5,0) to[short] (5,2.5);
\draw [](4,2.5) to[short] (5,3.5);
\draw [color=black](2.5,2.5) to[short] (5,5);
\draw [](4,-0.1) to[short] (4,0.1);
\fill[color={rgb,255:red,192; green,192; blue,192}]  (2.5,2.5) -- (5,5) -- (5,3.5) -- (4,2.5) -- cycle;
[
\end{circuitikz}
} & \resizebox{0.39\textwidth}{!}{%
\begin{circuitikz}
\tikzstyle{every node}=[font=\small]
\draw [->, >=Stealth] (0,0) .. controls (2,0) and (4,0) .. (6.3,0);
\draw [->, >=Stealth] (0,0) .. controls (0,5) and (0,5) .. (0,6.3);
\node [font=\normalsize] at (-1,6) {$\frac{1}{q}$};
\node [font=\small] at (6,-0.5) {$\frac{1}{p}$};
\node [font=\normalsize] at (-1,5) {$1$};
\node [font=\small] at (5,-0.5) {$1$};
\node [font=\small] at (-1,2.5) {$\frac{1}{2}$};
\node [font=\small] at (2.5,-0.5) {$\frac{1}{2}$};
\node [font=\small] at (-1,-0.5) {$0$};
\node [circle,fill,inner sep=1pt, color=red] at (2.5,2.5) {};
\node [font=\small] at (4,-0.5) {$-\frac{\theta}{n}+\frac{1}{2}$};

\draw [color=red, dashed](5,2.5) to[short] (5,5);
\draw [color=red, dashed](2.5,2.5) to[short] (5,2.5);
\draw [](0,5) to[short] (5,5);
\draw [](5,0) to[short] (5,2.5);
\draw [color=red](4,2.5) to[short] (5,4.5);
\draw [color=red](2.5,2.5) to[short] (5,5);
\draw [](4,-0.1) to[short] (4,0.1);
\fill[color={rgb,255:red,255; green,204; blue,204}]  (2.5,2.5) -- (5,5) -- (5,4.5) -- (4,2.5) --cycle;
[
\end{circuitikz}}
        \end{tabular}
    \end{figure}
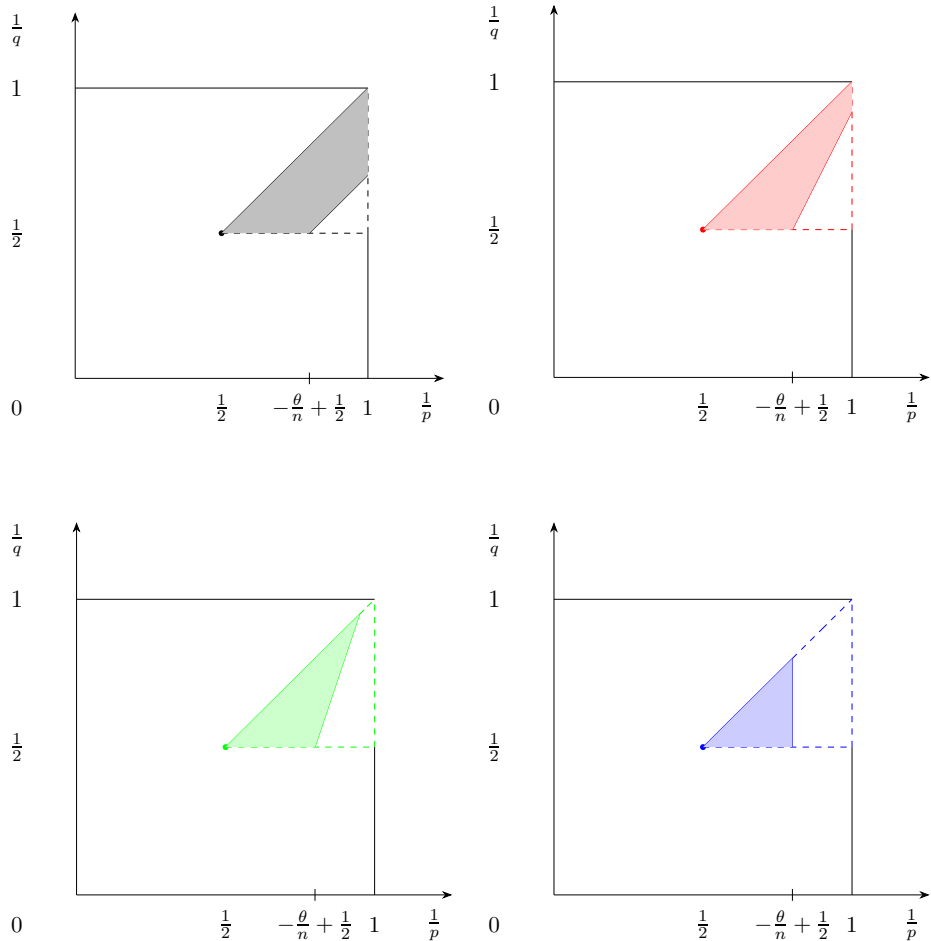
\begin{figure}[!ht]
\begin{tabular}{cc}
\resizebox{0.39\textwidth}{!}{%
\begin{circuitikz}
\tikzstyle{every node}=[font=\small]
\draw [->, >=Stealth] (0,0) .. controls (2,0) and (4,0) .. (6.3,0);
\draw [->, >=Stealth] (0,0) .. controls (0,5) and (0,5) .. (0,6.3);
\node [font=\normalsize] at (-1,6) {$\frac{1}{q}$};
\node [font=\small] at (6,-0.5) {$\frac{1}{p}$};
\node [font=\normalsize] at (-1,5) {$1$};
\node [font=\small] at (5,-0.5) {$1$};
\node [font=\small] at (-1,2.5) {$\frac{1}{2}$};
\node [font=\small] at (2.5,-0.5) {$\frac{1}{2}$};
\node [font=\small] at (-1,-0.5) {$0$};
\node [circle,fill,inner sep=1pt, color=green] at (2.5,2.5) {};
\node [font=\small] at (4,-0.5) {$-\frac{\theta}{n}+\frac{1}{2}$};

\draw [color=green, dashed](5,2.5) to[short] (5,5);
\draw [color=green, dashed](2.5,2.5) to[short] (5,2.5);
\draw [](4,-0.1) to[short] (4,0.1);
\draw [](0,5) to[short] (5,5);
\draw [](5,0) to[short] (5,2.5);
\draw [color=green](4,2.5) to[short] (4.75,4.75);
\draw [color=green](2.5,2.5) to[short] (4.75,4.75);
\draw [color=green, dashed](4.75,4.75) to[short] (5,5);
\fill[color={rgb,255:red,204; green,255; blue,204}]  (2.5,2.5) -- (4.75,4.75) -- (4,2.5) -- cycle;
[
\end{circuitikz}} &
\resizebox{0.39\textwidth}{!}{%
\begin{circuitikz}
\tikzstyle{every node}=[font=\small]
\draw [->, >=Stealth] (0,0) .. controls (2,0) and (4,0) .. (6.3,0);
\draw [->, >=Stealth] (0,0) .. controls (0,5) and (0,5) .. (0,6.3);
\node [font=\normalsize] at (-1,6) {$\frac{1}{q}$};
\node [font=\small] at (6,-0.5) {$\frac{1}{p}$};
\node [font=\normalsize] at (-1,5) {$1$};
\node [font=\small] at (5,-0.5) {$1$};
\node [font=\small] at (-1,2.5) {$\frac{1}{2}$};
\node [font=\small] at (2.5,-0.5) {$\frac{1}{2}$};
\node [font=\small] at (-1,-0.5) {$0$};
\node [circle,fill,inner sep=1pt, color=blue] at (2.5,2.5) {};
\node [font=\small] at (4,-0.5) {$-\frac{\theta}{n}+\frac{1}{2}$};

\draw [color=blue, dashed](5,2.5) to[short] (5,5);
\draw [color=blue, dashed](2.5,2.5) to[short] (5,2.5);
\draw [](4,-0.1) to[short] (4,0.1);
\draw [](0,5) to[short] (5,5);
\draw [](5,0) to[short] (5,2.5);
\draw [color=blue](4,2.5) to[short] (4,4);
\draw [color=blue](2.5,2.5) to[short] (4,4);
\draw [color=blue, dashed](4,4) to[short] (5,5);
\fill[color={rgb,255:red,204; green,204; blue,255}]  (2.5,2.5) -- (4,4) -- (4,2.5) -- cycle;
[
\end{circuitikz}}
\end{tabular}    

\caption{\label{grafica3}The values of $p$ and $q$ one can insert in inequality \eqref{pq} when $0<-\theta<\frac{n}{2}$, $1<p\leq q \leq 2$ and, $\rho=1$ in black, $\rho=\frac{1}{2}$ in red, $\rho=\frac{1}{3}$ in green, $\rho=0$ in blue.}
\end{figure}

\end{rem}
\section{Acknowledgements}
The authors were supported by the FWO Odysseus 1 grant G.0H94.18N: Analysis and Partial Differential Equations, the Methusalem programme of the Ghent University Special Research Fund (BOF) (Grant number 01M01021). MR is also supported by EPSRC grant EP/V005529/1, and FWO Senior Research Grant G011522N.

\section{Conflicts of interest}
On behalf of all authors, the corresponding author states that there is no conflict of interest.

\end{document}